\documentclass[10pt, english, a4paper]{extarticle}
\usepackage[left=3cm,right=3cm,top=2cm,bottom=2cm]{geometry}
\usepackage[full]{textcomp}
\usepackage[english]{babel}
\usepackage{amsmath,amssymb,mathrsfs,amsthm,stmaryrd}
\usepackage{mathtools} 
\usepackage{enumerate,enumitem}
\usepackage[hidelinks]{hyperref}
\usepackage{esint} 
\usepackage{upgreek} 
\usepackage[toc,page]{appendix}
\usepackage{todonotes}
\theoremstyle{plain}
\newtheorem{theorem}{Theorem}[section]

\newtheorem{lemma}[theorem]{Lemma}
\newtheorem{proposition}[theorem]{Proposition}
\newtheorem{assumption}[theorem]{Assumption}
\newtheorem{corollary}[theorem]{Corollary}
\theoremstyle{definition}
\newtheorem{example}[theorem]{Example}
\newtheorem{definition}[theorem]{Definition}
\newtheorem{remark}[theorem]{Remark}

\newtheorem*{idea}{Idea of proof}

\makeatletter

\@addtoreset{equation}{section}
\makeatother
\makeatletter
\newcommand*\bigcdot{\mathpalette\bigcdot@{.5}}
\newcommand*\bigcdot@[2]{\mathbin{\vcenter{\hbox{\scalebox{#2}{$\m@th#1\bullet$}}}}}
\makeatother
\newcommand{\bb}[2]{\ensuremath{(#1\bigcdot#2)}}
\newcommand{\R}{{\mathbb{R}}} 
\newcommand{\T}{\ensuremath{{\mathbb T}}} 
\newcommand{\X}{{\mathbf{X}}}
\newcommand{\Y}{{\mathbf{Y}}} 
\newcommand{\dd}{d}

\newcommand{\A}{{\mathbf A}}
\newcommand{\les}{\lesssim} 
\newcommand{\LL}{{\mathbb L}}

\title{Unbounded rough drivers, rough PDEs and applications}
\author{Antoine Hocquet \\ Technical University Berlin, Germany 
	\and Martina Hofmanov\'a \\ Bielefeld University, Germany
	\and Torstein Nilssen \\ University of Agder, Norway}
\date{\today}

\begin{document}
\maketitle

 \begin{abstract}
	A summary of recent contributions in the field of rough partial differential equations  is given. For that purpose we rely on the formalism of ``unbounded rough driver''. We present applications to concrete models including Landau-Lifshitz-Gilbert, Navier-Stokes and Euler equations.
\end{abstract}

\tableofcontents

\section{Introduction}
This review article aims to gather and explain some recent results obtained for rough partial differential equations  (RPDEs) with linear rough input, in connection with the so-called \textit{variational approach} introduced in \cite{bailleul2017unbounded,deya2019priori}. A central tool for that purpose is the concept of ``unbounded rough driver'' which is an object analogous to a rough path but possibly taking values in some space of unbounded operators.

The variational approach for rough PDEs is motivated in particular by filtering theory or stochastic conservation laws.
Parabolic rough PDEs of the following form are considered:
\begin{equation}
\left\{
\begin{aligned}
\label{ansatz_RPDE}
&du_t + G_t(x,u_t,\nabla u_t,\nabla^2u_t)dt = H_i(u_{t},\nabla u_t)d\X_t^i(x) ,\quad \text{on}\enskip [0,T]\times\R^d\,,
\\
&u_0\in L^2(\R^d),
\end{aligned}
\right.
\end{equation}
(with Einstein's summation convention over repeated indices),
where each \( H_i\) is affine linear and various assumptions on the coefficients $G$, $H_{i}$ are in order. 
Herein \( t\mapsto \X_t(\cdot)\) denotes a suitable enhancement of the coefficient path \( t\mapsto X_t(\cdot) \) which, seen as a path from \( [0,T] \) to a well-chosen functional space (for the space variable \( x\in \R^d \)), is \( \alpha \)-H\"older regular with \( \alpha>\frac13 \). 
A possible though non-generic choice is that of a transport multiplicative rough input where the time and space variables are split, e.g.
\( H_i(u,\nabla u)=\partial_i u ,\) \( i=1,\dots ,d, \) in \eqref{ansatz_RPDE} and
\[
X_t^i(x)=\sigma_\mu^i(x)Y_t^\mu, 
\]
(see Section \ref{sec:specific_URD}) and in that case we also write
\[
H_i(u_{t},\nabla u_t)d\X^i_t(x)=\sigma_\mu^i(x)\cdot\partial_i u_t d\Y^\mu_t.
\]
Here \( (Y^\mu)_{1\le \mu\le m} \) is the first level of an \( \alpha \)-H\"older rough path in \( \R^m \) and \( (\sigma _\mu)_{1\le \mu\le m}\) are vector fields.
In contrast with the time regularity which is allowed here to be low,
the existing literature assumes that the map \( x\mapsto X_t(x) \) has enough regularity (\( t\in [0,T] \) being frozen) which in turn imposes regularity on the vector fields. For all practical purposes, it would be enough to assume bounded differentiability up to order \( 3 \). To avoid cumbersome discussions about indices and to simplify the statements, we choose here to work with spatially smooth coefficients \( X_t(\cdot) \), \( \sigma(\cdot) \).

\subsection*{Frequently used notations}
For Banach spaces \( E,F \), we write \( E\hookrightarrow F\) if \( E \) is continuously embedded in \( F. \)
We write \( L(E,F) \) for the space of bounded linear operators \( T\colon E\to F \) equipped with the usual induced norm \( |T|_{L(E,F)}=\sup_{|f|_E\le1}|Tf|_{F} .\)  
For an open set \( U\subset\R^d \) we denote by \( L^p(U;E) , p\in [1,\infty]\), the usual Lebesgue spaces of equivalence classes of functions from \( U\to E \). Accordingly we let \( \mathscr S'(U,E) \) be the space of $E$-valued Schwartz distributions on \( U \).

Throughout we assume that \( T>0 \) is a finite time-horizon and we introduce the simplices
\[
\begin{aligned}
	&\Delta (c,d) \vcentcolon=\{(s,t)\in[0,T]^2,\enskip c\le s\leq t\le d\},
	\\
	&\Delta _2(c,d)\vcentcolon=\{(s,\theta,t)\in[0,T]^3,\enskip c\leq s\leq \theta \leq t\leq d\}\,,
\end{aligned}
\]
for every \( [c,d]\subset[0,T] \).
We will use the abbreviations \( \Delta=\Delta(0,T)\) and \( \Delta_2=\Delta_2(0,T). \)
If \( v\colon [0,T]\to E\) is a continuous path, we denote its increment by \( \delta v\colon \Delta\to E \), that is
\begin{equation}
	\label{nota:increments}
	\delta v_{st}\vcentcolon=v_t-v_s,
\end{equation} 
for every \( (s,t)\in\Delta \).
For \( z=(z_{st})\colon \Delta\to E \) we also define \( \delta z \colon \Delta_2\to E\) as the 3-parameter quantity
\begin{equation}
\delta z_{s\theta t}=z_{st}-z_{s\theta}-z_{\theta t},
\end{equation}
for each \( (s,\theta,t)\in \Delta_2 \).
For $\alpha\in (0,1)$, we denote by \( C^\alpha=C^\alpha([0,T];E) \)  the space of \( \alpha \)-H\"older continuous paths from \( [0,T]\to E \). By \( C_2^{\alpha}=C_2^{\alpha}([0,T];E) \) we denote the set of mappings  \( z:\Delta \to E \) such that
$$
[z]_{\alpha;E}:=\sup_{s,t\in \Delta, s\neq t}\frac{\|z_{s t}\|_{E}}{|t-s|^{\alpha}}<\infty.
$$
We use the same notation $[z]_{\alpha;E}$ or simply $[z]_{\alpha}$ for the usual  H\"older semi-norm in $C^\alpha=C^\alpha([0,T];E)$.

\section{Unbounded rough drivers}

Herein we assume that a family of Banach spaces \( E=(E_\beta,|\cdot|_{\beta})_{\beta\in \R}\) is given such that \( E_{{\gamma}}\hookrightarrow E_{{\beta}} \) for \( \beta\le\gamma \in \mathbb R. \)%
\footnote{It is not necessary to choose the whole real line as index set (a well-chosen subinterval would suffice), however this turns out to simplify the presentation. Accordingly, we will work with coefficients that are  smooth in space, with all derivatives being bounded.}
We suppose for the sake of presentation that each \( E_\beta \) for \(\beta\in \R\) is continuously embedded in the space of Schwartz distributions on \( \R^d \) with values in \( \R^{k} \)  for some integers \( d,k\ge1 \) (so as to include systems of equations). 

\subsection{Formal definition}

Loosely speaking, an unbounded rough driver is a rough path whose values are not real numbers or $\mathbb{R}^{m}$ vectors but rather linear operators. We want to include the case of unbounded operators, which are described by a negative shift of indices in the corresponding scale, and this explains why we cannot work on a single Banach space. To simplify our settings, it is convenient to assume some additional compatibility properties between the spaces \( E_{\beta} \) for \( \beta\in \R \). The following interpolation assumption, which will be granted throughout the manuscript, is satisfied on a number of classical examples (see Example \ref{exa:scales}).
For instance, it is fulfilled if there is a family of smoothing operators on \( E \) in the sense of \cite{hocquet2018energy,hofmanova2019navier}.
\begin{assumption}[Interpolation between intermediate spaces]
	\label{ass:intermediate}
 The scale \( (E_\beta)_{\beta\in \R}\) satisfies the following property for any \( \beta\le \iota\le\gamma \in \R \):
for each \( f\in E_\iota \) and every \( \eta>0 \) small enough, there is a decomposition
\begin{equation}
\label{decomp_tilde}
f=f^\eta + \tilde f^\eta \,
\in\,
E_\gamma + E_\beta
\end{equation}
such that 
\begin{equation}
\label{decomp_quant}	
|f^\eta|_{\gamma}\les \eta^{-(\gamma-\iota)}|f|_{\iota}
\quad \text{while}\quad 
|\tilde f^{\eta}|_{\beta} \les \eta^{\iota-\beta} |f|_{\iota}.
\end{equation}
\end{assumption}

As it turns out, properties \eqref{decomp_tilde}-\eqref{decomp_quant} are equivalent to the continuous embedding
\begin{equation}
	\label{intermediate_spaces}
	E_\iota \hookrightarrow
	(E_\beta,E_\gamma)_{\frac{\iota-\beta}{\gamma-\beta},\infty}
\end{equation}
where the right most term denotes the real interpolation space of order \( (\theta,q)=(\frac{\iota-\beta}{\gamma-\beta},\infty) \) between the Banach spaces \( X=E_{\beta} \) and \( Y=E_\gamma \) (see \cite{lunardi2009interpolation}). 
These are satisfied by quite a number of functional spaces, as illustrated above.

\begin{example}\label{exa:scales}
It is well-known that \eqref{intermediate_spaces} (and hence \eqref{decomp_tilde}-\eqref{decomp_quant}) is satisfied, for instance if 
\begin{itemize}
\item \( E_\beta = B^{\beta}_{p,q} \) is the usual Besov scale with \( p,q\in [1,\infty]\).
This includes the Sobolev scale \(  E_\beta =W^{\beta,p} \) by making the choice \( p=q\).
\item \( E_{\beta}=H^{\beta,p}\) is the Bessel potential scale.
\item \( E_\beta =D(L^{\beta/2})\) where \( L \) is a maximal accretive operator with bounded imaginary powers (the factor \( 1/2\) here is only for convenience).
\end{itemize}
We refer e.g.\ to \cite{triebel1983theory} for definitions.
\end{example}

For an operator \( T\) from \( \mathscr S'(\R^d,\R^k) \) to itself and a scale as above (recall that \( E_\beta\hookrightarrow \mathscr S'(\R^d,\R^k)  \), \( \forall\beta\in \R \)), we write
\[ T\in L(E) \]
if the operator $T$ acts on $E$ in the following sense: for any \( \beta\in \R \) there is another \( \beta_1\in \R \) such that \( T\in L(E_\beta,E_{\beta_1}).\)
Fix two real numbers \( \sigma\ge0  \), \( \alpha\in (\frac13,\frac12]\) and consider an operator-valued path \( A \colon [0,T]\to L(E)\) such that
\[
|A_t-A_s|_{L(E_\beta,E_{\beta-\sigma})} \lesssim _\beta (t-s)^\alpha\text{ for every } (s,t)\in \Delta.
\]
As it is the case for differential equations driven by rough paths, we cannot address \eqref{ansatz_RPDE} based on the sole knowledge of $(A_t)_{t\in [0,T]}$ (or equivalently on its increments \( A^1_{st}:=\delta A_{st} \) if \( A_0=0 \)). Rather, we need to assume that an ``enhancement'' \( \A=(A^1,A^2) \) thereof is given, containing consistent additional data, namely,  the iterated integral of \( A \). 
This leads us to the next definition.

\begin{definition}[Unbounded rough driver on \( E \)]
	\label{def:URD}
	Fix $\alpha \in (\frac13,\frac12]$ and  \( \sigma\in[0,1] .\)
	
	\noindent A $2$-index family $\A_{st}\equiv(A^1_{st},A^2_{st})_{(s,t)\in\Delta }$ of linear operators in $E$ is called an \emph{unbounded rough driver} of \( \alpha \)-H\"older regularity provided
	\begin{enumerate}[label=(\roman*)]
		\item \label{RD1}
		for $i=1,2,$ \( \beta\in \R ,\) and every \( (s,t)\in \Delta \), the operator
		$A^i_{st}$ is bounded from \( E_{\beta}\to E_{\beta-i\sigma} \) and satisfies the uniform bound
		\begin{equation}
		\label{bounds:rough_drivers}
		|A^i_{st}|_{L(E_{\beta},E_{\beta-i\sigma})}\les_{\beta} (t-s)^{i\alpha}\,.
		\end{equation}
		\item \label{RD2}
		Chen's relations hold true, namely, for every $(s,\theta ,t)\in\Delta _2,$ we have in the sense of linear operators:
		\begin{equation}\label{chen}
		\delta A^1_{s\theta t}=0\,,\quad \delta A^2_{s\theta t}=A^1_{\theta t}A^1_{s\theta } \,.
		\end{equation}
	\end{enumerate}
We call $\A$ \emph{geometric} if a sequence of smooth paths $A(n),n\geq 0,$ exists such that 
 for every \( \beta\in \R\)
	\begin{align}
	\label{rho_alpha}
	[\A(n)-\A]_{\alpha,\beta}:=
	\|A^1(n)-A^1\|_{C^\alpha_2 (0,T;L(E_{\beta},E_{\beta-\sigma}))}
	+[A^2(n)-A^2]_{C^{2\alpha}_2(0,T;L(E_{\beta},E_{\beta-2\sigma}))}\to0.
	\end{align}
Here for each \( n\ge0 \) we have denoted by $\A(n)$ the \textit{canonical lift} corresponding to
\( A^1_{st}(n):= A_{t}(n)-A_{s}(n) \), with the level-two component
\[
A^{2}_{st}(n):= \int_s^t  d A^1_r(n)A_{sr}^1(n)\,.
\]
\end{definition}

\subsection{Specific examples}
\label{sec:specific_URD}

Herein we let \( (E_\beta)_{\beta} \) be any of the scales in Example \ref{exa:scales}, and we recall that \( E_\beta\hookrightarrow \mathscr S'(\R^d,\R^k) \) for some  integers \( d,k\ge1. \)

\subsubsection{Operator-valued paths of order one}
\label{sec:diff}
We start with the case of differential operators acting on scalar functions.

\paragraph{The case of differential operators}
Fix \( k=1 \) and assume that the first level is given by a differential operator of order one, that is
\begin{equation}
\label{B1_diff}
 A^1_{st}= X^i_{st}(x) \partial_i + X^0_{st}(x)
\end{equation}
for some coefficient path $t\mapsto X_t(x)\in \R^{d+1}$ which is smooth in space and $\alpha$-H\"older in time (by the latter, we mean that \( |X_{st}(\cdot)|_{C_b^\beta}\lesssim_\beta (t-s)^{\alpha} \) for each \( \beta\ge0 \)). This is clearly enough to ensure that \eqref{bounds:rough_drivers} holds for \( i=1 \), but the case \( i=2 \) is a bit less clear since there is some freedom in the definition of the operators \( A^2_{st} \).

\subparagraph{Geometric enhancement property}
Despite the lack of a canonical meaning for the iterated integral of \( t\mapsto A_t \) with itself, geometricity in the sense of Def.~\ref{def:URD} imposes further additional constraints on the second level of $\A$.
This is illustrated by the next lemma, whose proof can be found in \cite{hocquet2020ito,hocquet2021quasilinear}.

\begin{lemma}
	\label{lem:geometric}
Let \( \A \) be a geometric enhancement of \eqref{B1_diff}. 
There is a \( 2 \)-index family of coefficents $\bb{X}{\nabla X^i}_{i=0,1,\dots,d}$ which is subject to the Chen's-type relation
\begin{equation}
	\label{chen_L}
	\delta \bb{X}{\nabla X^i}_{s\theta t}(x)=\partial _j X_{s\theta }^i(x)X_{\theta t}^j(x) \,,\quad (s,\theta,t)\in\Delta_2,
\end{equation} 
for each \( x\in \R^d \), \( i=0,1,\dots ,d \),
such that the second level of \( \A \) is given by
\begin{equation}
\label{weak_geo}
\begin{aligned}
	A^2_{st}
&:= \tfrac12X^i _{st}X^j _{st}\partial _{ij} + \bb{X}{\nabla X^i} _{st}\partial _i + \bb{X}{\nabla X^0}_{st} + \frac12(X^0_{st})^2\,.
	\end{aligned}
\end{equation}	
\end{lemma}
Note that the symbol \( \bb{X}{\nabla X} \) must be thought of as a placeholder for the a priori ill-defined quantity ``\( \int_s^t dX_{r}\cdot  \nabla \delta X_{s,r} \)'', \( (s,t)\in \Delta \). In this sense, it is part of the required additional data to make sense of level-two and give \eqref{ansatz_RPDE} a proper meaning.

Interestingly enough, the relation \eqref{weak_geo} yields the following identity for the ``non-commutative bracket'' 
\[
2A^2_{st}-(A^1_{st})^2\equiv(2\bb{X}{\nabla X}_{st}-X_{st}\cdot\nabla X_{st})\cdot\nabla 
\]
which is not zero in general, even though \( \A \) was obtained as a limit of canonical enhancements. This effect of ``non-commutativity of the coordinates'' highlights the constrast between geometric unbounded rough drivers and geometric rough paths (for a real-valued geometric rough path \( \mathbf Y \in \mathscr C^\alpha_g(\mathbb R)\), it is indeed always true that \( 2Y^2_{st}-(Y^1_{st})^2 \) vanishes identically).

\begin{example}
	\label{example:bracket}
	Let $(\delta Y^\mu,\mathbb Y^{\mu\nu})_{1\leq \mu,\nu\leq m}$ be an $m$-dimensional, two-step geometric rough path of H\"older regularity $\alpha>\frac13$ and suppose that \( X^0=0 \) while
	\[ 
	X^i_{st}(x)=\sigma^i_\mu(x) \delta Y^\mu_{st},\quad i=1,\dots ,d,
	\]
for given smooth and bounded coefficients \( \sigma^i_{\mu}\). 
	Then we have
\[
\bb{X}{\nabla X^i}_{st}:= \mathbb Y_{st}^{\mu \nu} \partial_j\sigma^i_\mu \sigma^j_\nu\,.
\]
If we introduce the \textit{L\'evy area} $\mathbb L_{st}=\frac12(\mathbb Y - \mathbb Y^{\star})$, we see using geometricity that
	\begin{equation}\label{id_bracket}
	A^2_{st}= \tfrac12(A^1_{st})^2 + \mathbb L^{\mu \nu}_{st}[\sigma_\nu,\sigma_\mu]\cdot\nabla.
	\end{equation} 
	where $[\cdot ,\cdot]$ is the usual Lie bracket for vector fields. 
	If the vector fields $\sigma_\mu^{\cdot}(x)\in \R^d$ commute for \( \mu=1,\dots ,m \), the above relation tells us that the non-commutative bracket is zero, or equivalently that \(A^2\equiv(A^1)^2/2. \) Hence, in this specific case there is no need for additional information: we can make sense of \eqref{ansatz_RPDE} without any additional structure.
\end{example}

The algebraic identities of Example \eqref{example:bracket} are still valid if we assume instead that the rough path \( \Y \) is only weakly geometric.
Recall that a rough path $(\delta Y^\mu,\mathbb Y^{\mu\nu})_{1\leq \mu,\nu\leq m}$  of H\"older regularity $\alpha > \frac13$ 
 is called \emph{weakly geometric} 
 if it satisfies the symmetry condition
$$
\mathbb{Y}^{\mu \nu}_{st} + \mathbb{Y}^{\nu\mu}_{st} = 2\delta Y^{\mu}_{st} \delta Y^{\nu}_{st}
$$
for all \( (s,t)\in \Delta \). 
From \cite{FrizVictoir} we know that there exists a sequence of smooth paths $Y(n): [0,T] \rightarrow \R^m$ such that its canonical lift $(\delta Y^{\mu}(n), \mathbb Y^{\mu\nu}(n))_{1\leq \mu,\nu\leq m}$ converges to $(\delta Y^\mu,\mathbb Y^{\mu\nu})_{1\leq \mu,\nu\leq m}$ in the $\alpha'$-rough path metric for every $\alpha'\in (\frac13, \alpha)$.

\subparagraph{Weak versus strong URD geometricity}
It was noted in \cite{CN} that an unbounded rough driver can also be constructed from an infinite-dimensional rough path as follows. Assume that 
$
X: [0,T] \rightarrow C^3_b(\R^d;\R^d)
$
lifts to an infinite-dimensional rough path in the sense that there exists 
$$
\mathbb{X} : \Delta  \rightarrow C^3_b(\R^d \times \R^d; \R^{d \times d})
$$
such that 
$$
\delta \mathbb{X}_{s \theta t}^{ij}(x,y) = X_{s \theta}^i(x) X_{\theta t}^j(y) .
$$
Denote by $\nabla_x^{\otimes} : C^3_b(\R^d \times \R^d; \R^{d \times d}) \rightarrow C^2_b(\R^d \times \R^d; \R^d)$ is the closure of the operator defined on tensors $f \otimes g$ by
$$
(\nabla_x^{\otimes} (f \otimes g))^j (x,y) = g^i(y) \partial_i f(x)^j. 
$$
It is then readily checked that 
\begin{equation} \label{eq:inf_dim to URD}
A^1_{st} := X_{st}^i(x) \partial_i , \qquad A^2_{st} := (\nabla_x^{\otimes} \mathbb{X}_{st})^j(x,x) \partial_j + \mathbb{X}^{ij}_{st}(x,x) \partial_i \partial_j
\end{equation}
(for simplicity we leave out the multiplicative term $X^0$ of \eqref{B1_diff}) defines an unbounded rough driver on the scales defined in Example \ref{exa:scales}. In fact, we see that 
$$
\bb{X}{\nabla X}_{st} = \nabla_x^{\otimes} \mathbb{X}_{st}
$$ 
where $\bb{X}{\nabla X}$ is as in Lemma \ref{lem:geometric}.

Although the seemingly irrelevant second spatial variable, $y$, is not used to construct the unbounded rough driver in \eqref{eq:inf_dim to URD} it offers a satisfying infinite-dimensional results on the analogue of the notion of weakly geometric rough paths; we say that $(X,\mathbb{X})$ is weakly geometric if 
\begin{equation} \label{eq:tensor symmetry}
\mathbb{X}^{ij}_{st}(x,y) + \mathbb{X}^{ji}_{st}(y,x) = X_{st}^i(x) X_{st}^j(y).
\end{equation}
In \cite{GNS} it was proved that Hilbert-space valued rough paths satisfies the same approximation property as is expected from the finite-dimensional setting. The following result follows from \cite[Theorem 1.1]{GNS}.
\begin{proposition}
Assume $(X,\mathbb{X})$ is an $W^{k,2}$-valued $\alpha$-rough path for $k > \frac{d}{2} + 4$ and satisfies \eqref{eq:tensor symmetry}. Then there exists a sequence of paths $X(n) : [0,T] \times \R^d \rightarrow \R^d$ which are of bounded variation in $t$ and $W^{k,2}$ in $x$ such that the canonical rough path lift $(X(n), \mathbb{X}(n))$ converges to $(X,\mathbb{X})$ in the $C^3_b$-valued $\beta$-rough path metric for every $\beta \in (\frac13, \alpha)$. 
\end{proposition}

\begin{remark}
Note that in the above result (relatively) high spatial regularity is required to have continuous embedding into $C^3_b$ of the second component $\mathbb{X}$. It is clear that the construction \eqref{eq:inf_dim to URD} is continuous and as such, gives a geometric unbounded rough driver in the sense of \eqref{rho_alpha}, albeit with $\beta$ in place of $\alpha$. 
\end{remark}

\paragraph{Incompressible Navier-Stokes}
Here we introduce the URD corresponding to the \( d \)-dimensional (\( d=2,3 \)) rough Navier-Stokes equations (see Sec.~\ref{sec:NS} below). We assume that the velocity component is divergence free, which leads us to introduce the Leray projection \( P\colon \mathscr S'\to \mathscr S' \) onto divergence-free vector fields.
If we work with periodic distributions,
 the latter is defined by means of Fourier transform 
\( Pf = \sum_{k\in \mathbb Z^d}(\hat f_n-\frac{n\cdot \hat f_n}{n^2})e_n\) where \( e_n(x)=(2\pi)^{-\frac{d}{2}}e^{in\cdot x} \) and \( f_n=\langle f,e_n\rangle \).
The resulting equation forms a system on the pair \( (u, \pi)\) where \( \pi_t= \int _0^t\nabla p_rdr\) (here \( p \) is the related pressure term) and therefore we work with \( k=2d \).
In this setting, the case of transport noise may be described as follows:
we introduce
\[
\begin{aligned}
&A^{P,1}_{st}=Y_{st}^\mu [P \circ(\sigma_\mu\cdot \nabla)] ,
\quad \quad 
A^{P,2}_{st}=\mathbb Y_{st}^{\mu\nu} [P \circ(\sigma_\nu\cdot \nabla)\circ P \circ(\sigma_\mu\cdot \nabla)] ,
\\
&A^{Q,1}_{st}=Y_{st}^{\mu} [Q \circ(\sigma_\mu\cdot \nabla)]
\quad \quad 
A^{Q,2}_{st}=\mathbb Y_{st}^{\mu\nu} [Q \circ(\sigma_\nu\cdot \nabla)\circ P \circ(\sigma_\mu\cdot \nabla)] ,
\end{aligned}
\]
where \( Q=I-P \) and \( \mathbf Y =(Y,\mathbb Y)\) as in Example \ref{example:bracket}.
Next, we introduce the URD
\[
A^1_{st}= 
\begin{pmatrix}
	A^{P,1}_{st} & 0\\
	A^{Q,1}_{st} & 0
\end{pmatrix}
\quad \quad 
A^2_{st}=
\begin{pmatrix}
	A^{P,2}_{st} & 0\\
	A^{Q,2}_{st} & 0
\end{pmatrix}
\]
and we see from the immediate relation \( \delta A^{Q,2}_{s\theta t}=A^{Q,1}_{\theta t}A^{P,1}_{s\theta} \) that Chen's relations are indeed satisfied for \( \mathbf A, \) that is
\[
\delta A^2_{s\theta t}=
\begin{pmatrix}
	A^{P,1}_{\theta t}A^{P,1}_{s\theta} & 0\\
	A^{Q,1}_{\theta t}A^{P,1}_{s\theta} & 0
\end{pmatrix}
=\begin{pmatrix}
	A^{P,1}_{\theta t} & 0\\
	A^{Q,1}_{\theta t} & 0
\end{pmatrix}
\begin{pmatrix}
	A^{P,1}_{s\theta} & 0\\
	A^{Q,1}_{s\theta} & 0
\end{pmatrix}
=A^{1}_{\theta t}A^{1}_{s\theta}
\]
(the analytical condition \eqref{bounds:rough_drivers} is clear).

\subsubsection{Operator-valued paths of order zero}
\label{sec:anti}
We now give some comments on the case when \( \sigma=0 \) (i.e.\ the underlying path takes values in a subalgebra of bounded operators) and \( k\ge1 \) is arbitrary (as opposed to the previous section where \( \sigma=1 \)). 
More specifically, we suppose that a path \( X\colon [0,T]\to L(E) \) is given for each \( t\in [0,T] \), taking the form of a multiplication operator
\[
f\in E\subset \mathscr S'(\R^d,\R^k)
\mapsto 
[X_tf](x)= \sum_{1\le i,j\le k}X_t^{ij}(x)f^i(x)
\]
with coefficients \( X^{ij}_t(x) \) that are smooth with respect to \( x \) and \( \alpha \)-H\"older continuous in \( t \).

We still assume geometricity (in the sense given in Definition \ref{def:URD}): namely \( X \) can be obtained as the limit, for the inherited $C^\alpha(L(E))\times C_{2}^{2\alpha}(L(E))$ topology, of a sequence of smooth rough drivers \( \X (n)=(X^1(n),X^2(n))\), defined by canonical enhancement of \( X(n) \).

We record a simple lemma which is used in \cite{gussetti2023pathwise} about the antisymmetric case.
		Herein we let $L_a(\R^k)\subset L(\R^k)$ be the space of linear maps $T\colon \R^k\to \R^k$ such that $T v\cdot w = -v\cdot(T w)$ $\forall v,w\in\R^k$.

\begin{lemma}
	\label{lem:antisymmetric}
Suppose that \( X_t(x) \) takes values in \( L_a(\mathbb R^k) \).
		There exists family of coefficients 
		\[
 \LL=(\LL_{st}(x))\in C^{2\alpha}_2 \left ([0,T];C^\infty_b(\R^d,L_a(\R^k))\right )
\]
 such that
		\begin{equation}\label{levy}
		X^2_{st}(x) = \frac12(X^1_{st}(x))^2 +\LL_{st}(x)\,.
		\end{equation}
\end{lemma}		

\begin{proof}
Assuming that $X^{1}$ has finite variation,
		an easy integration by parts argument over $r\in[s,t]$, using the antisymmetry of $X_t$, implies 
		\[
		X^{2,\star}_{st}=\left (\int_s^t dX_{r}\delta X_{s,r}\right )^{\star}=
		\int_s^t \delta X_{sr}dX_{r} = \int _s ^tdX_r\delta X_{rt}\,.
		\]
		Summing, we obtain that the symmetric part of $X^2_{st}$ is determined by the first level $X^1_{st}=\delta X_{st}$ through
		\[
		\frac12\left (X^2_{st}+ X^{2,\star}_{st}\right ) = \frac12(\delta X_{st})^2\,.
		\]
		Since geometric rough drivers are obtained as limits of such finite-variation lifts, this entails a similar identity at the limit, hence our conclusion.
\end{proof}		

		In the particular  case \( k=3 \) one sees that $\LL_{st}$ is in fact the L\'evy area of $\X$, formally:
		\begin{equation}
		\label{bracket}
		\begin{aligned}
		\LL_{st} =
		\iint\limits_{s<r_1<r<t}
		\frac12
		\begin{pmatrix}
		0 &
		d X_{r}^2 d X_{r_1}^1 -  d X_{r}^1 d X_{r_1}^2 &
		d X_{r}^3 d X_{r_1}^1- d X_{r}^1 d X_{r_1}^3
		\\
		d X_{r}^1 d X_{r_1}^2 - d X_{r}^2 d X_{r_1}^1 &
		0 &
		d X_{r}^3 d X_{r_1}^2- d X_{r}^3 d X_{r_1}^2
		\\
		d X_{r}^1 d X_{r_1}^3 - d X_{r}^3 d X_{r_1}^1&
		d X_{r}^2 d X_{r_1}^3 - d X_{r}^3 d X_{r_1}^2&
		0
		\end{pmatrix}.
		\end{aligned}
		\end{equation}

\section{A priori estimates}
We explain how to obtain two types of a priori estimates for rough PDEs of the form
\begin{equation}
\label{general_RPDE}
d v_t = G_t(x,\nabla v_t,\nabla ^2v_t) d t + (d\X_t \cdot \nabla  + d\X_t^0)v_t 
\quad\text{on }[0,T]\,.
\end{equation}
where the unknown \( v\colon [0,T]\to E_0 \) is bounded and \( G \) is a reasonable nonlinearity (see below).
The first type of estimates takes the form of various \( p \)-variation bounds that apply to remainder terms in the corresponding Euler-Taylor expansion. 
It only makes use of the linear structure of the rough term and it should be highlighted that few properties on the drift coefficient are needed (see the first paragraph below).
The second part shows how to start from the RPDE \eqref{general_RPDE} and obtain estimates for the Sobolev type norms \( \|v\|_{L^2(0,T;W^{1,2}) \cap L^{\infty}(0,T;L^2)}.\) The latter requires ellipticity of \( G \), even though it is possible to generalize the argument to the (transport) case when \( G \) is independent of \( \nabla^2v \) (in which case one obtains an estimate on \( \|v\|_{L^{\infty}(0,T;L^2)} \)).

\subsection{Controlled rough paths estimates}
We are now interested in remainder estimates for equations of the form
\begin{equation}
\label{ansatz_Q}
d v = F d t + d \A v\,,
\quad\text{on }[0,T]
\end{equation} 
for a distributional drift $F$ that may depend on \( v \). 

We assume furthermore that \( \|v\|_{L^\infty(0,T;E_{0})}=\sup_{t\in [0,T]}|v_t|_0 <\infty \) and that \( \A \) is an URD as in Definition \ref{def:URD}, where for convenience we suppose
that \( \sigma=1 \). 
The way \eqref{ansatz_Q} is understood is via a suitable Euler-Taylor expansion.

\begin{definition}[Solution]
\label{def:diff_notation}
A path $v:[0,T]\to E_{0}$ is called a \textit{solution} of \eqref{ansatz_Q} provided:
\begin{itemize}
\item 
\( F\) belongs to  \( L^1(0,T;E_{-2})\);
\item 
The following expansion holds as an equality in $E_{-2}$ for every $(s,t)\in\Delta $: 
\begin{equation}
\label{euler-taylor_Q}
\delta v_{st} =\int_s^tF_rd r + (A^{1}_{st} + A^2_{st})v_s+ v^\natural_{st}\,,
\end{equation}
In the above formula, the first integral is a Bochner one, 
and the remainder \( v^\natural \) (implicitly defined via \eqref{euler-taylor_Q}) 
has the property that $(s,t)\mapsto|v^\natural_{st}|_{-3}$ has finite \( q \)-variation for some \( q<1 \).
\end{itemize}
Moreover, we say that $u$ starts at \( u^0\in E_0 \) if
\begin{itemize}
	\item 
	 with probability one, \( \lim _{t\downarrow0}u_t \) exists weakly in \( E_0 \) and equals \( u^0 \).
\end{itemize}
\end{definition}

The next statement is a generalization of a result that appears in various forms in the literature (see for instance \cite{deya2019priori,hocquet2018energy,hofmanova2019navier,hocquet2020ito}).
It says that given the ansatz \eqref{ansatz_Q}, we can estimate the \( p \)-variation norm  for \( p\in(1, \frac{1}{3\alpha}] \) of the remainder \( v^{\natural} \) in terms of the corresponding quantities for \( F \) and \( \A \). 

\begin{proposition}[Remainder estimates]
	\label{pro:apriori}
	Let $v\in L^\infty(0,T;E_0)$ be a solution of \eqref{ansatz_Q}, in the sense of the Euler-Taylor expansion \eqref{euler-taylor_Q},
	for some drift coefficient $F\in L^1(0,T;E_{-2})$ and let 
	\[
	w_F(s,t):=
	\int_s^t|F_{r}|_{-2} d r\,.
	\]

	There are positive constants $C,L$ depending only on \( \alpha ,[\A]_\alpha ,\|F\|_{L^1(E_{-2})} \)
	such that for each $(s,t)\in\Delta $ subject to	$|t-s|\leq L,$
	it holds 
	\begin{equation}
	\label{estimate_remainder}
	|v^{\natural}_{st}|_{-3}
	\leq C \left((t-s)^{3\alpha }\| v\|_{L^\infty(s,t;E_0)}
	+(t-s)^\alpha w_F(s,t)\right)\,,
	\end{equation}
	and
	\begin{equation}
	\label{estimate_remainder_2}
	|R^v_{st}|_{-2}+|v^{\natural}_{st}|_{-2}
	\leq C \left((t-s)^{2\alpha }\| v\|_{L^\infty(s,t;E_0)}
	+w_F(s,t)\right).
	\end{equation}
where \[R^v_{st}:= \delta v_{st}-A^1_{st}v_s\,.\]
Moreover, we have the following estimate on the increment for \( 0\le t-s\le L \):
	\begin{equation}
	\label{estimate_remainder_3}
	|\delta v_{st}|_{-1}\leq C \left(w_F(s,t)^{\alpha } + (t-s)^\alpha \|v\|_{L^\infty(s,t;E_0)}\right)\,.
	\end{equation}
\end{proposition}
Summing the above estimates for different time-intervals, one can easily infer some \( p \)-variation estimates for \( \delta v ,R^v, v^{\natural}\)  (we omit the details).
Let \( \beta\in \R \) and assume for notational convenience that \( p=\frac1\alpha\in [2,3)\). 
Recall that the \( p \)-variation of a \( 2 \)-index map \( z=(z_{st})\colon \Delta\to E_\beta \) on \( [u,v]\subset [0,T] \) is  defined as the quantity
\[
\|z\|_{\alpha,\beta;[u,v]}:= \left (\sup_{\pi}\sum_{[t_i,t_{i+1}]\in\pi} |z_{t_{i},t_{i+1}}|_{\beta}^{\frac1\alpha}\right )^{\alpha}
\]
where the supremum is taken over all  finite partitions \( \pi=\{[t_0,t_1],[t_1,t_2],\dots ,[t_n,t_{n+1}]\} \) such that \( t_0=u \) and \( t_{n+1}=v\).
\begin{proof}
	Note the identity
	\[
	R^v_{st} := \delta v_{st} - A^1_{st}v_{s}= \int_s^tF_r d r + A^2_{st}v_{s} + v_{st}^\natural\,.
	\]
	Using Assumption \ref{ass:intermediate}, we can interpolate these two different expressions for $R^v,$ by writing
	\begin{equation}
	\label{Rv_estim}
	\begin{aligned}
	|R^v_{st}|_{-2}
	&\leq |(\int_s^t F_r d r  + A_{st}^2v_s + v_{st}^\natural)^{\eta}|_{-2} + |[\delta v_{st}-A^1_{st}v_s]^{\sim,\eta}|_{-2}
	\\
	&\lesssim |\int_s^tF_rd r|_{-2} + |A^2_{st}v|_{-2} + \frac{|v^\natural_{st}|_{-3}}{\eta}
	\\
	&\quad \quad \quad \quad \quad 
	+\eta ^22\|v\|_{L^\infty(E_0)} +\eta (t-s)^\alpha  \|v\|_{L^\infty(E_0)}.
	\end{aligned}
	\end{equation} 
	
	In order to estimate $v^\natural,$ note that Chen's relations \eqref{chen} imply
	\[
	\delta v^{\natural}_{s\theta t}
	=A^1_{\theta t}R^v_{s\theta } + A^2_{\theta t}\delta v_{s\theta }\,,\quad \text{for}\enskip (s,\theta ,t)\in\Delta _2\,.
	\]
	From this and the sewing lemma (Theorem \ref{thm:sewing}), we infer that
	\begin{equation}
	\label{pre:estim:v_nat}
	|v^\natural_{st}|_{-3}\les_{\alpha,[\A]_\alpha}
\left((t-s)^{\alpha }\|R^v\|_{2\alpha,-2;[s,t]}^{2\alpha} + (t-s)^{2\alpha }\|\delta v\|_{\alpha,-1;[s,t]}^{\alpha}\right) \,.
	\end{equation}
	
	Now, since \eqref{Rv_estim} is true for arbitrary $\eta \in(0,1),$ we can choose $\eta :=\zeta (t-s)^\alpha $ for some $\zeta >0$ big enough. We obtain from \eqref{pre:estim:v_nat}:
	\[
	\begin{aligned}
	|R^v_{st}|_{-2}
	&\les \left(\int_s^t|F_r|_{-2}d r\right) 
	+ (t-s)^{2\alpha }\|v\|_{L^\infty(s,t;E_0)}
	+\frac{\|v^\natural\|_{3\alpha,-3;[s,t]}^{3\alpha}}{\zeta (t-s)^\alpha }
	\\
	&\leq C\left(w_F(s,t)
	+ (t-s)^{2\alpha }\|v\|_{L^\infty(s,t;E_0)}\right) 
	\\
	&\quad \quad \quad \quad \quad 
	+ \frac12\left( \|R^v\|_{2\alpha,-2;[s,t]}^{2\alpha} +(t-s)^{2\alpha } \|\delta v\|_{\alpha,-1;[s,t]}^\alpha \right),
	\end{aligned}
	\]
	provided that $(t-s)\leq L(\alpha ).$
	Using that the right hand side is a control, we infer from simple facts about \( p \)-variation estimates that
	\[
	\|R^v\|_{\alpha,-2;[s,t]} \lesssim
w_F(s,t)
	+ (t-s)^{2\alpha }\|v\|_{L^\infty(s,t;E_0)}
	+(t-s)^{2\alpha } \|\delta v\|_{\alpha,-1;[s,t]}^\alpha \,.
	\]

Next, similar considerations lead to 
\[
	\|\delta v\|_{\alpha,-1;[s,t]} \lesssim
w_F(s,t)^\alpha
	+ (t-s)^{\alpha }\|v\|_{L^\infty(s,t;E_0)}
	+(t-s)^\alpha  \|R^v\|_{2\alpha,-2;[s,t]}^{2\alpha }\,.
	\]
We can  then close the argument by looking again at \eqref{pre:estim:v_nat}: from the two previous estimate, we infer that \eqref{estimate_remainder} holds. The rest of the statement is immediate.
\end{proof}

\subsection{Behavior with respect to products}

In what follows, we consider a fixed open set
$
U\subset \R^d
$
and work for convenience with the Sobolev scale \( E_{\beta}=W^{\beta,p}(U) \) 
for some \( p\in[1,\infty] \) (see Example \ref{exa:scales}).
Suppose that we are given two bounded paths \( t\mapsto u_t\) and \( t\mapsto v_t \), from \([0,T]\to L^p(U) \) and \([0,T]\to L^{p'} (U)\)  respectively (with $1/p+1/p'=1$). Assume that both paths satisfy an equation of the form \eqref{ansatz_Q} with the same URD \( \A \) but with possibly different drift terms.
If the regularity conditions of Proposition \ref{pro:apriori} are satisfied for both paths,
it is then natural to expect that the pointwise product $uv$ will be still described by a relation of the form \eqref{ansatz_Q}. However, unless the rough term only contains transport terms, the zeroth order of the corresponding URD must be shifted from a factor \( 2 \).
This is described in the next proposition which establishes a product formula in the case of differential URDs of order \( \le 1 \).
Note that analogous formulas may be derived for systems i.e.\ when \( k\ne1 \) (see, e.g., \cite{gussetti2023pathwise}).
 
\begin{proposition}[Product formula, the case of differential operators]
	\label{pro:product}
	Let $(X^{i},\bb{X}{\nabla X^i})_{i=0,\dots,d}$ %
	be the coefficients of a geometric URD as in Sec.~\ref{sec:diff}.
	Fix $p,p'\in [1,\infty]$ with $1/p+1/p'=1,$ and consider two bounded paths as above 
	such that
	\[
\begin{aligned}
	& d u= F d t +  (d\X\cdot\nabla + d\X^0)  u,
	\\
	& d v=G  d t + (d\X\cdot\nabla + d\X^0) v,
	\end{aligned}
	\]
	for some $F\in L^p(0,T;W^{-1,p}(U))$ and \( G\in L^{p'}(0,T;W^{-1,p'}(U)) \).

Suppose further that
	\[u\in L^p(0,T;W^{1,p}(U)),\quad 
	  v\in L^{p'}(0,T;W^{1,p'}(U))
	\]
Then,  $uv$ is subject to the regularity conditions on Proposition \ref{pro:apriori} (with \( p=1 \) and \( E_\beta=W^{\beta,1} \) therein).
		It is a solution 
		of
		\begin{equation}
		\label{concl:prod}
		 d (uv)= \big[uG +Fv\big] d t +  (d\X \cdot\nabla + 2d\X^0)(uv)\,.
		\end{equation}
	Here the second component of
\( \tilde A^1_{st}=X^i_{st}\partial_i+2X^0_{st} \) is understood as
\[
\tilde A^{2}_{st}
:=X^i_{st}X^j_{st}\partial_{ij}+ \bb{X}{\nabla X^i}_{st}\partial_i + 2X^0_{st}X^i_{st}\partial _i+2\bb{X}{\nabla X^0}_{st}+2(X^0_{st})^2 \,,
\]
which makes it a geometric URD.	
\end{proposition}

\subsection{The main Sobolev (energy) estimates}
\label{sec:monotone}

Suppose that we are given a bounded path \( u \colon [0,T]\to L^2(\R^d)\)
as in \eqref{ansatz_Q} with \( F=-G(u) \), that is
\begin{equation}
\label{monotone_eq}
\left \{\begin{aligned}
	d u= -G(u) d t +  (d\X\cdot\nabla + d\X^0)  u
	\\
	u_0\in L^2(\R^d),
\end{aligned}\right .
\end{equation}
for a drift term \( G(u) \) subject to the monotonicity condition
\begin{equation}
	\label{monotone}
	 \langle G(u),u\rangle \ge0,\quad \quad \forall u\in W^{1,2}(\R^d).
\end{equation}
In a classical variational setting where \( t\mapsto X_t \) has finite variation and \eqref{monotone_eq} is understood weakly (i.e.\ via test functions), the assumption \eqref{monotone} can then be exploited to obtain a priori estimates. In that case, one simply needs to justify that the space of admissible test functions can be extended so as to include the solution \( u \) itself. Testing \eqref{monotone_eq} against \( u \) will then result in a differential relation satisfied by \(t\mapsto|u_t|_{L^2}^2 \), which, along with Gronwall Lemma, will yield the main energy estimate:
\begin{equation}
	\label{energy_bound}
	\sup_{t\in [0,T]}|u_t|_{L^2}^2 + 2\int_{0}^{T}\langle G(u_t),u_t\rangle dt \le C|u_0|_{L^2}^2
\end{equation}
for some constant \( C\ge0 \) depending on the \( 1 \)-variation of \( X .\)

\subsubsection*{Strategy in the rough setting}
In the URD context, the correct analogue of ``testing the equation against the solution'' is not so straightforward. 
The starting point is the product formula from Proposition \ref{pro:product} which asserts in particular for \( p=2 ,u=v\) (assuming the ansatz \( F=-G(u)\)), that
\begin{equation}
	\label{eq:u2}
	d (u^2)= -2uG(u)d t +  (d\X \cdot\nabla + 2d\X^0)(u^2) \,.
\end{equation}
The latter is in fact (still) an equation, with the unknown \( v=u^2 \) being bounded as a path from \( [0,T]\to L^1(\R^d) \). Accordingly, the additional information \eqref{monotone} will only be helpful once we integrate that equation -- or what is the same -- if we test it against the constant \( \phi\equiv1 \). Precisely, from \eqref{euler-taylor_Q} we find that 
\begin{multline}
	\label{ineq:u2}
\delta (|u_{\cdot}|_{L^2}^2)_{st} + 2\int_s^t\langle G(u_r),u_r\rangle d r  
= \Big\langle\Big(X^i_{st}\cdot\partial_i + 2 X^0_{st} 
\\
+X^i_{st}X^j_{st}\partial_{ij}+ \bb{X}{\nabla X^i}_{st}\partial_i + 2X^0_{st}X^i_{st}\partial _i+2\bb{X}{\nabla X^0}_{st}+2(X^0_{st})^2
\Big)(u^2_s)\,,\, 1\Big\rangle
+ \langle u^{2,\natural}_{st},1\rangle
\\
 \les |u_s|^2_{L^2}(t-s)^\alpha + |\langle u^{2,\natural}_{st},1\rangle|
\end{multline}
where to estimate the first bracket term in the right hand side we have used integration by parts.
If a similar estimate were true for the second term remainder term \( |\langle u^{2,\natural}_{st},1\rangle| \) then we could conclude from a variant of Gronwall Lemma (Lemma \ref{lem:gronwall}) that \eqref{energy_bound} indeed holds true.
For that purpose we note from the first remainder estimate in Proposition \ref{pro:apriori} and the fact that \( 1\in E_3 \):
\[
|\langle u^{2,\natural}_{st},1\rangle|
\les
|v^{\natural}_{st}|_{-3}
\leq C \left((t-s)^{3\alpha }\sup_{r\in [0,t]}|u_r|^2_{L^2}
+(t-s)^\alpha w_{G}(s,t) dr\right)
\]
where \( w_{G}(s,t)=\int_s^t|G_{r}(u_r)u_r|_{-2}\).

If we further assume that the second monotonicity condition
 \begin{equation}
 	\label{monotone2}
 	 w_{G}(s,t)\les \int_s^t\langle G(u_r),u_r\rangle dr 
 \end{equation}
is satisfied
(which happens to be the case of many instances of \eqref{general_RPDE} under appropriate ellipticity for \( G \)) then we may absorb the latter term to the left hand side of \eqref{ineq:u2}, so as to end up in the same pre-gronwall scenario as in the classical setting and thereby show \eqref{energy_bound}.

\begin{remark}
	\label{rem:monotonicity}
	The argument above still works if the monotonicity condition \eqref{monotone} is replaced by the weaker statement that
	\begin{equation}
		\label{monotone_bis}
		\langle G(u)+\lambda u,u\rangle \ge0,\quad \quad \forall u\in W^{1,2}(\R^d)
	\end{equation}
	for some constant \( \lambda\ge0\), and similar for condition \eqref{monotone2}. We obtain in this case an analogue of \eqref{energy_bound}, with a constant \( C \) depending on \( \lambda \).
\end{remark}
\section{Parabolic rough PDEs}

Herein we work in the setting of Section \ref{sec:monotone}, assuming that \( G(u)=-L_tu \) for some uniformly elliptic operator. 
In that case it can be shown that the conditions of Remark \ref{rem:monotonicity} are satisfied.

\subsection{Linear case}
In \cite{hocquet2018energy,hocquet2020ito}, the linear case with
\[
L_tu= \partial _i(a^{ij}_t(x)\partial _ju) + b^i_t(x)\partial _iu + c_t(x)u
\]
was treated.
Here we are given non-degenerate measurable coefficents $a^{ij}$ (bounded above and below), and $b$ and $c$ satisfy appropriate integrability conditions. These ensure that the monotonicity properties of Remark \ref{rem:monotonicity} are statisfied.
In that setting, we consider
\begin{equation}
\label{rough_paths_eq}
d u_t-(L_tu +f_t(x))d t= (d \X_t \cdot\nabla +d\X^0_t)u_t\enskip,
\quad \text{on}\enskip (0,T]\times \R^d\,,
\end{equation}
which is amenable to the strategy introduced in Section \ref{sec:monotone} (we assume geometricity of the corresponding URD).
It is understood in the sense of Definition \ref{def:diff_notation}, i.e.\ as the Euler-Taylor expansion \eqref{ansatz_Q} with the URD \(\A\) built as in Sec.~\ref{sec:diff}. As usual we assume that the unknown \( u\) is bounded as a path from \( [0,T]\to L^2 \), but also that its energy
\[
\|u\|^2_{L^\infty(0,T;L^2)\cap L^2(0,T;W^{1,2})}:=\sup_{t\in [0,T]}|u_t|_{L^2}^2 + \int _0^T|\nabla u_r|_{L^2}^2dr
\]
is finite. It is not too hard to show that the drift  \( F=Lu +f \) then satisfies \( \int_0^T|F_r|_{-2}dr<\infty \), as required in Definition \ref{def:diff_notation}. Furthermore, this is also enough to guarantee that the regularity hypotheses of propositions \ref{pro:apriori} and \ref{pro:product} on the drift term are fulfilled.

In contrast with other methods such as the mild approach (see \cite{gubinelli2010rough}), the regularity of the coefficients in the left hand side of  \eqref{rough_paths_eq} can be lowered down to the usual minimal assumptions in the deterministic context (see, e.g.\ \cite[Chap. 4]{ladyzhenskaya1968linear}), for instance we only require that \( f\in L(0,T;W^{-1,2}) \).

The main solvability results are summarized as follows.
\begin{theorem}
	\label{thm:HH17}
	For the linear equation \eqref{rough_paths_eq}, given $u_0\in L^2(\R^d),$ there exists a unique solution $u$ in the class \( L^\infty(0,T;L^2)\cap L^2(0,T;W^{1,2}) \). It is global in time,
	and depends continuously%
	\footnote{For the weak topology of \( L^\infty(0,T;L^2)\cap L^2(0,T;W^{1,2}) \). A similar statement for the strong topology was shown in \cite{FNS}}
	on the rough enhancement $(X,\bb{X}{\nabla X})$ and the coefficients \( a,b,c,f \).

Moreover if \( X^0=0 \), the solution \( u_t(x) \) satisfies the following chain rule:
\begin{equation}
\label{ito_intro}
d (\beta(u_t))-\beta'(u_t)(L_tu_t+f_t)d t= d \X\cdot \nabla(\beta(u_t)).
\end{equation}
for any $\beta\in C^2(\R,\R)$ such that \( |\beta''|_{C_b}<\infty \) and \( \beta(0)=0 \).
\end{theorem}

Note that in transport theory, the property \eqref{ito_intro} is usually referred to as ``renormalization'' \cite{diperna1989ordinary}.
The problem of writing a chain rule for \eqref{rough_paths_eq} arises in a very natural way when studying its well-posedness, as illustrated in Section \ref{sec:monotone} and the search for an energy estimate (which corresponds to composition with $\beta(z)=\frac{z^2}{2}$).
Having \eqref{ito_intro}  is also useful to establish comparison principles, where a classical argument (due to Stampacchia) makes use of (suitable) regularized versions of the function $\beta(z)=\max(z-M,0)$  for \( M\in\R \) large enough. We refer to \cite{hocquet2015landau} for details.

We now sketch the proof of \eqref{ito_intro} as given in \cite{hocquet2020ito}.
	
\begin{idea}[Chain rule]
If $u$ is bounded, it is possible to apply \eqref{concl:prod} inductively to obtain the chain rule on polynomials, which by a density argument allows one to conclude.
This property is then extended to \emph{any} solution, thanks to the observation that a sufficiently large class of rough parabolic equations have bounded solutions.
Boundedness among the latter class is the core of the proof, and it is obtained by an extension of the celebrated Moser iteration argument to the rough context \cite{moser1964harnack}.
\end{idea}

\subsection{Non-linear variants in the transport noise case}

\subsubsection{Semi-linear case}
Motivated by a version of a rough, viscous Burgers equation, and in order to circumvent some technical difficulties arising with semi-linear equations, another approach (though related to \cite{hocquet2018energy} in some aspects) was introduced in  \cite{hocquet2020generalized}.
Letting $L_t$ be the Laplacian (this is however not essential), one considers the problem of solving \textit{semi-linear} equations of the form
\begin{equation}
	\label{rough_burgers}
	d u-\Delta u d t=\mathrm{div} F(u) d t+ \nabla u \cdot d \X\quad \text{on}\enskip [0,T]\times\R^d\,.
\end{equation}
with \( u_0\in L^2 \) and under appropriate assumptions on $F\colon \R\to\R$.
This could be achieved in principle by obtaining (local in time) energy estimates.
If $d=1,$ and $F(u)=-\frac12u^2$ (in which case \eqref{rough_burgers} becomes a Burgers-type equation with transport rough input), we need to circumvent the ``linear nature'' of rough Gronwall Lemma, which is not helpful here.
In the general case $d\geq 1,$ this can be done by introducing a new unknown $m\colon[0,T]\times \R^d\to \R,$ solving the backward equation
\begin{equation}
	\left \{
	\begin{aligned}
	d m=-\Delta mdt+\mathrm{div} (m d\X)\,,
	\\
	m_T(x)\equiv 1\,.
	\end{aligned}
\right .
\end{equation}
which is a kind of continuity equation (by analogy with the pure transport case when \( L_t\equiv0 \)).
The solution $m$ can be explicitly computed using a Feynmann-Kac representation (see for instance \cite{diehl2017stochastic}).
By testing the equation \eqref{rough_burgers} against $um$ (which can be justified for solutions of finite energy, i.e.\ with $\|\nabla u\|_{L^2([0,T]\times \R^d)}<\infty$) one obtains the following ``weighted'' energy equality
\begin{equation} \label{EnergyEquality}
	\langle u_t^2, m_t\rangle + \int_0^t 2\langle |\nabla u_r|^2, m_r\rangle  dr 
	= |u_0|^2_{L^2} - 2 \int_0^t \langle F(u_r) , \nabla( u_r m_r)\rangle dr \,.
\end{equation}
A remarkable property of equation \eqref{EnergyEquality} is that it contains no rough term.
It can then be used, together with a suitable Bihari-Lassalle inequality, in order to obtain the existence of a sufficiently small $T>0$ such that $u$ has indeed finite energy.

\subsubsection{Quasilinear case}
Based partly on the latter weighted energy estimate technique and \eqref{ito_intro}, in \cite{hocquet2021quasilinear} a quasilinear evolution equation with transport noise of the form 
\begin{equation}
\label{problem}
\begin{aligned}
d u = \mathrm{div}(a(u)\nabla u)d t +  \nabla u\cdot d\X ,\quad [0,T]\times\mathbb{T}^d,
\\
u_0\in L^2(\mathbb T^d)
\end{aligned}
\end{equation} 
was treated.
Here we let
$a\colon \R\to L(\R^d,\R^d)$ be $C^1,$ bounded above and below.  
The following result was obtained.

\begin{theorem}[\cite{hocquet2021quasilinear}]
	Let $a\colon \R\to \mathscr L(\R^d,\R^d)$ be $C^1,$ bounded above and below. There exists a solution to \eqref{problem}. If $\mathrm{div} X_t(\cdot)=0,\forall t$ or if \( d=1 \) and
	\[
	\nabla u\in L^p(0,T;L^q(\mathbb T))
	\text{ with }\frac1p+\frac{1}{2q}<\frac12,
	\]
	then this solution is unique.
\end{theorem}

\begin{idea}[Uniqueness]
	The proof of uniqueness when \( d=1 \) and \( \nabla X\neq 0 \), which is the main difficulty in the statement, goes as follows.
	If $u^1,u^2$ are two solutions of the same Cauchy problem, denote by $A_t$ the self-adjoint operator
\[
L_t\varphi:= \nabla\cdot(a(t,x,u^1) \nabla\varphi )\,,\quad \varphi \in W^{1,2}\,,
\]
and observe that $v:=u^1-u^2$ solves the equation
\[
\left\{\begin{aligned}
	& d v = \big[L_tv + \nabla\cdot(b_t(x)v)]dt
+ d\A v
	\\
	&v_0=0\,,
\end{aligned}\right.
\]
where
\[
b_t(x):= \mathbf 1_{v_t(x)\neq 0}\frac{a(t,x,u^1_t(x))-a(t,x,u^2_t(x))}{v_{t}(x)}\nabla u^2_t(x)
\]
Using that $a$ is Lipshitz implies the integrability property
\[
b\in\ L^{2r}(0,T;L^{2q})\quad \text{with}\enskip
\frac{1}{r} + \frac{1}{2q}<1\,.
\]
Let $m_t$ denote a non-negative solution of
\begin{equation}
	\label{m_intro}
	\left\{
	\begin{aligned}
		&\dd m  + (Lm -b\cdot \nabla m)\dd t =(\dd \X\cdot \nabla +\nabla\cdot(\dd\X))m\,,\quad \text{on}\enskip [0,T]\times \T\,,
		\\
		&m_T(x)=1\,,
	\end{aligned}\right.
\end{equation} 
	and consider an appropriate sequence $\beta_n(z)$ of $C^2_b$ approximations of the modulus function $z\mapsto|z|$.
	Applying the chain rule \eqref{ito_intro}, using the weighted energy estimate \eqref{EnergyEquality} and letting \( n\to\infty \) entails the following \( L^1 \)-type estimate on \( v=u^1-u^2 \):
	the following weighted inequality
	\[
	\int _{\T}|v_t|m_t \dd x \leq 0\,.
	\] 
	 The rest of the proof is devoted to showing that \( \inf _{t,x}m_t(x)>0 \), from which we infer uniqueness.
	\hfill\qed
\end{idea}

\subsection{Non-local transport noise}

In \cite{CN} a non-local version of the rough PDE is considered, i.e.
\begin{equation} \label{eq:Fokker-Planck}
d\rho_t = \frac12 \text{Tr}(\nabla^2 b[t,\rho_t]b[t,\rho_t]^T \rho_t) - \text{div}(\sigma_{\mu}[\rho_t] \rho_t) d \mathbf{Y}_t^{\mu}.
\end{equation}
In the above equation the solution $\rho$ should be considered to be measure-valued and in fact, the above is the Fokker-Planck equation of a non-linear/non-local diffusion
\begin{equation} \label{eq:non-local lagrangian}
dx_t = b_{\nu}(t,x_t, \rho_t) dB^{\nu}_t + \sigma_{\mu}(x_t,\rho_t) d\mathbf{Y}_t^{\mu}
\end{equation}
where $\rho_t = \mathcal{L}(x_t)$ is the law of the solution. Here, $b_{\nu} : \R^d \times \mathcal{P}(\R^d) \rightarrow \R^d$ is Lipschitz w.r.t. to both variables (the measure components is given the Wasserstein-metric) and 
$$
\sigma_{\mu}(x,\rho) := \int_{\R^d} \sigma_{\mu}(x,\xi) d \rho(\xi).
$$

To understand \eqref{eq:Fokker-Planck}, let us first fix the solution, $\rho: [0,T] \rightarrow \mathcal{P}(\R^d)$. We can then interpret \eqref{eq:Fokker-Planck} as an equation where the noisy vector field is given by
$$
X_{st}(x) = \int_s^t \sigma_{\mu}(x, \rho_r) d\mathbf{Y}_r^{\mu}, \qquad \mathbb{X}_{st}(x,y) = \int_s^t \int_s^u \sigma_{\mu}(x,\rho_r)  \sigma_{\nu}(y,\rho_u) d\mathbf{Y}_u^{\nu} d \mathbf{Y}_r^{\mu} ,
$$
and the URD is given by the construction in \eqref{eq:inf_dim to URD}.

\section{Stochastic Landau-Lifshitz-Gilbert equation}

We treat the Landau-Lifshitz-Gilbert equation driven by a linear multiplicative noise via the URD formalism. This permits to obtain useful properties resulting from the continuity of the It\^o-Lyons solution.
In the sequel we fix a stochastic basis \( (\Omega,\mathcal F,\mathbb P,\{\mathcal F_t\}) \) subject to the usual right-continuity and completeness assumptions and we let \( W\colon \Omega\times[0,T]\to H \) be a \( Q \)-Wiener process for a well chosen Hilbert space \( H \) and trace-class operator \( Q\in L(H) \) (see \cite{DPZ}).

We work on the one-dimensional torus \( \mathbb T=\R/\mathbb Z \) and we let \( H^\alpha=W^{\alpha,2}(\mathbb T;\R^3) \) for \( \alpha\in \R\).
The system of equations we are concerned with is the following \textit{Stratonovich} SPDE with unknown \( u\colon \Omega\times [0,T]\times \mathbb T\to \R^3\):
	\begin{equation}
	\label{LLG}
	\left\{\begin{aligned}
	& d u-(\Delta u+u|\nabla  u|^2 + u\times\Delta u) d t=u\times\circ d W \,,
	\quad \text{on}\enskip [0,T]\times\mathbb T 
	\\
	&|u_t(x)|=1\,,\quad \quad \forall (t,x)\in [0,T]\times\mathbb T \,,
	\\
	&u_0\in H^{1}(\mathbb T;\R^3)\,,
	\end{aligned}\right . 
	\end{equation}
where \( a\times b \) denotes vector product in \( \R^3 \) and here we take \( H=L^2(\mathbb T, \R^3) \) for the  Hilbert space in which \( W_t(\omega) \) lies.
The system \eqref{LLG} is a well-known simplified model for the dynamics of magnetized materials subject to thermal fluctuations.
A concrete motivation to study \eqref{LLG} within the URD framework is the study of small noise asymptotics.
This problem is particularly important in micromagnetism to study phase transition between different equilibrium states of the magnetization (see for instance \cite{hocquet2015landau} and the references therein).

As is easily observed (at least when \( W \) is spatially smooth), for every \( t\in [0,T]\) and for \( \mathbb P \)-a.e.\ \(\omega \), the map \( X_t(\omega)=(\cdot)\times W_{t}(\omega) \) belongs to the subclass of anti-symmetric, multiplicative operators as introduced in Section \ref{sec:anti} with \( k=3 \).
This leads us to consider an anti-symmetric (pathwise random) URD $\omega\rightarrow \mathbf{X}(\omega)$ and we think of it as a pathwise enhancement for the linear operator \( u\mapsto u\times W_t(\omega) \) (we assume for simplicity that its spatial dependence is smooth).
\begin{definition}
		\label{def:solution}
		We say that a stochastic process $u\colon \Omega\times [0,T]\to L^2(\T;\R^3)$  is a \textit{pathwise solution} of the \eqref{LLG} provided that \[\mathbb P\Big(u\in L^\infty (0,T;H^1)\cap L^2(0,T;H^2) 
		\quad \&\quad 
		 |u_t(x)|=1\,,\,\,\forall (t,x)\in [0,T]\times\mathbb T\Big)=1,\]
			and there exists \( q<1 \), a random variable $u^{\natural}\colon \Omega\to C^{q\mathrm{-var}}_{2,\mathrm{loc}}(0,T;L^2)$ such that
			\begin{equation}
			\label{rLLG_def}
			\delta u_{st}-\int_s^t(\Delta u_r +u_r|\nabla  u_r|^2 +u_r\times\Delta u_r)  d r
			=\delta X_{st}u_s + (\tfrac12(\delta X_{st})^2 +\LL_{st})u_s + u^\natural_{st}\,,
			\end{equation}
			in $L^2(\mathbb T ;\R^3)$, for every $s\leq t\in[0,T] $ and $\mathbb P$-a.s., where \( \LL \) is the L\'evy Area of \( \X \) as in Lemma~\ref{lem:antisymmetric}.
	\end{definition}

The following result was shown in \cite{gussetti2023pathwise}.

\begin{theorem}
		\label{thm:existence}
		Let $u^0\in H^k$  for some $k\ge 1$ and suppose that \( |u^0_t(x)|\equiv 1 \).
		There is a unique pathwise solution starting from \( u^0 \) with \(\mathbb P( u\in \mathcal X^k:=L^\infty(H^k)\cap L^2(H^{k+1}) )=1\).	
		It has finite moments of arbitrary order in \( \mathcal X^k\) and satisfies the Stratonovich equation:
		\[
		\begin{aligned}
		du_t 
		=(\Delta u_t +u_t|\nabla u_t|^2 + u_t\times\Delta u_t)dt + u _t\times \circ dW \quad \text{ on }\,\Omega\times[0,T]\times\T,\qquad u_0=u^0.
		\end{aligned}
		\]
		Morerover, there is a set of full measure \( \Omega_1 \) and a continuous, deterministic solution map
		\( \pi\) from \(H^k(\T ;\mathbb{S}^2)\times (C^\alpha\times C^{2\alpha}_2) \longrightarrow L^\infty(H^1)\cap L^2(H^2)\)
	 	such that for any \( \omega \in \Omega_1\)
	 	\[
	 	 u(\omega)= \pi(u^0,\X(\omega))\,.
	 	\]
Precisely, it satisfies the pathwise local Lipschitz property
		\begin{align}
		\|\pi(u^0,\mathbf{X})-\pi(\tilde u^0,\mathbf{\tilde X})\|_{L^\infty(H^1)\cap L^2(H^2)}\lesssim |u^0 - \tilde u^0|_{H^k}+[\mathbf{X}-\mathbf{\tilde X}]_{\alpha;C_b^{k+2}}\,.
		\end{align}
for an implicit constant that depends only on the size of \( u^0,\tilde u^0\in H^k ,\mathbf X,\mathbf{\tilde X}\) in the corresponding norms.
\end{theorem}

A useful consequence of the previous result is that if $u=\pi(u^0, h)$ is the solution driven by a Cameron-Martin path $h\in\mathcal{H}$ enhanced geometrically to an antisymmetric URD \( \mathbf X^{h} \), then $(\mathbb{P}\circ(u^\epsilon)^{-1})_{\epsilon>0}$ satisfies a large deviations principle in \( L^\infty(H^k)\cap L^2(H^{k+1}) \) with good rate function
		\begin{align}
		\mathcal{J}(y):= \inf _{h\in\mathcal{H}}\big(\mathcal{I}(h) : \pi(u^0,\X^h)=y\big)\, ,
		\end{align}
where we have denoted by \( u^\epsilon \) the solution obtained by the substitution of \( W\leftarrow \sqrt{\epsilon}W \) in \eqref{LLG} and \(\mathcal{I}(h):=
\int_{0}^{T}|\dot{h}_t|^2 dt\) if \( h\in \mathcal{H}\)
 (it equals \( +\infty \) otherwise).

Especially the Wong-Zakai result proved essential in establishing further results concerning the Landau-Lifshitz-Gilbert equation in the subsequent works \cite{G23,G22}. In \cite{G22}, existence (and in some cases also uniqueness) of an ergodic invariant measure was established. Here, the Wong-Zakai result implied the so-called Feller property which is not possible to prove by purely probabilistic estimates. The work \cite{G23} introduced a general framework to study pathwise central limit theorem and moderate deviations principle for a general class of stochastic partial differential equations perturbed with a small linear multiplicative noise. Also here, the rough path theory and particularly the framework of URD plays the key role.

\section{Navier-Stokes equations}
\label{sec:NS}

This line of research started in \cite{hofmanova2019navier} with the Navier-Stokes equations on the velocity level perturbed by a rough transport noise 

\begin{align}\label{nse:1}
\left\{
\begin{aligned}
du+(u\cdot\nabla u+\nabla p)d t &= \Delta u d t+ \sigma_{\mu}\cdot\nabla u d\mathbf Y^{\mu},\qquad \text{on }[0,T]\times\mathbb{T}^{d},\\
{\rm div} u &= 0,\\
u_{0}&\in H.
\end{aligned}
\right.
\end{align}
Here and the sequel $H$ denotes the space of $L^{2}$-integrable, mean and divergence-free vector fields, $u$ denotes the velocity of an incompressible fluid, $p$ is the associated pressure. We suppose as in Example \ref{example:bracket} that $\mathbf{Y}$ is a finite-dimensional, continuous geometric rough path of finite $p$-variation with $p\in [2,3)$ and that the coefficients $\sigma_{i}$ are divergence-free, bounded with  bounded derivatives up to order two. As discussed in the paragraph on Navier-Stokes equations in Section~\ref{sec:diff}, the system \eqref{nse:1} can be rewritten using the URD formalism applied to the couple of the velocity $u $ and the effective pressure $\pi:=\int_{0}^{\cdot}\nabla p_{r}d r$. However, similarly to the deterministic setting, it proves beneficial to work with the equation for $u$ and  $\pi$ separately. This leads us to the following definition of solution in the spirit of Definition~\ref{def:diff_notation}. Here, $P$ denotes the Leray projection, $Q=I-P$, $H^{n}$ and $H_{\perp}^{n}$, $n\in \mathbb{R}$, are the projected Bessel potential  spaces $H^{n}=P (I-\Delta)^{-n/2}L^{2}$, $H_{\perp}^{n}=Q (I-\Delta)^{-n/2}L^{2}$.

\begin{definition}\label{def:solution2}
A pair of weakly continuous functions $(u, \pi) : [0,T] \rightarrow H \times H_{\perp}^{-3}$ is called a solution of \eqref{nse:1} if $u\in L^2(0,T; H^1)\cap L^{\infty}(0,T; H)$ and the remainders
\begin{align} 
u_{st}^{P,\natural}& :=   \delta u_{st}  +  \int_s^t \left[P(u_{r}\cdot\nabla u_{r}) -\nu \Delta u_r  \right]\,dr   -  [A_{st}^{P,1} +A_{st}^{P,2}]u_s , \label{SystemSolutionU} \\
u_{st}^{Q,\natural}& :=    \delta \pi_{st}+ \int_s^t  Q(u_{r}\cdot\nabla u_{r})\,dr   -  [A_{st}^{Q,1}+A_{st}^{Q,2}] u_s , \label{SystemSolutionPi}
\end{align}
have finite $q$-variation for some $q<1$ in $H^{-3}$  and $H_{\perp}^{-3}$, respectively.
\end{definition}

Using the framework of URD, it was then possible to obtain the following result.

\begin{theorem}
Let $d=2,3$. There exists a solution to \eqref{nse:1} satisfying the energy inequality
$$
\|u_{t}\|_{L^{2}}^{2}+2\int_{0}^{t}\|\nabla u_{s}\|_{L^{2}}^{2} d s\leq \|u_{0}\|_{L^{2}}^{2},\qquad t\in[0,T].
$$
If $d=2$ and the vector fields $\sigma_{i}$ are constant, then the solution is unique, satisfies the energy equality and depends continuously in appropriate topologies on the initial condition, the rough path $\mathbf Y$ and the vector fields $\sigma_{i}$.
\end{theorem}

The restriction of constant vector fields was due to technical problems coming from the fact that the Leray projection onto divergence-free vector fields must be included in corresponding URD. This issue could be overcome when working on the level of vorticity $\xi=\mathrm{curl}u=\nabla\times u$. Particularly, the work \cite{HLN21} considered
\begin{align}\label{nse:2}
\left\{
\begin{aligned}
d\xi+(u\cdot\nabla \xi-\mathbf{1}_{d=3}\xi\cdot\nabla u)d t &= \Delta \xi d t+ (\sigma_{\mu}\cdot\nabla \xi - \mathbf{1}_{d=3}\xi\cdot\nabla\sigma_{\mu}) d\mathbf Y^{\mu},\qquad \text{on }[0,T]\times\mathbb{T}^{d},\\
\xi_{0}&\in H.
\end{aligned}
\right.
\end{align}
The velocity $u$ can be then be  recovered using the inverse of the curl, the so-called Biot-Savart operator. On the velocity level, \eqref{nse:2}  reads as
\begin{align}\label{nse:3}
\left\{
\begin{aligned}
du+(u\cdot\nabla u+\nabla p)d t &= \Delta u d t+ (\sigma_{\mu}\cdot\nabla u + \nabla\sigma_{\mu}u )d\mathbf Y^{\mu},\qquad \text{on }[0,T]\times\mathbb{T}^{d},\\
{\rm div} u &= 0,\\
u_{0}&\in H^{1}.
\end{aligned}
\right.
\end{align}

The main result was as follows.

\begin{theorem}
If $d=2,3$, then there exists a unique local strong solution to \eqref{nse:3}. If $d=2$, then the  solution is global and depends continuously on $(u_{0},\sigma,\mathbf Y)$.
\end{theorem}

We note that due to the higher regularity of the initial condition required for the vorticity form \eqref{nse:2}, the solution $u$ to \eqref{nse:3} in the above result is analytically strong. In dimension $d=3$ it exists only up to some possibly small time.

The technology developed in \cite{hofmanova2019navier,HLN21} was successfully employed in \cite{FHLN20} to obtain regularization of the Navier-Stokes equations by deterministic vector fields. In particular, the focus was on
\begin{align}\label{nse:4}
\left\{
\begin{aligned}
d\xi+(u\cdot\nabla \xi-\xi\cdot\nabla u)d t &= \Delta u d t+ \Pi(v\cdot\nabla \xi), \qquad\text{on }[0,\infty)\times\mathbb{T}^{3},\\
\xi_{0}&\in H,
\end{aligned}
\right.
\end{align}
where $\Pi$ denotes the Leray projection and  $v$ is a deterministic divergence-free vector field.
Similarly to \eqref{nse:2}, \eqref{nse:4} possesses a local unique solution. The additional term $\Pi(v\cdot\nabla \xi)$ is energy conservative due to the divergence-free constraint of $v$, namely, $\langle\Pi(v\cdot\nabla \xi,\xi\rangle=0$. Hence,  seemingly it does not influence the final time of existence, denoted by $\tau(\xi_{0})$, which is typically obtained from the energy inequality. Nevertheless, it is possible to choose $v$ in a way that the following holds true.

\begin{theorem}
Let $\mu$ be a probability measure on $H$. For any $\varepsilon>0$, there is a deterministic vector field $v$ such that
$$
\mu\big(\{\xi_{0}\in H;\tau(\xi_{0})=\infty\}\big)\ge 1-\varepsilon.
$$
In other words, the collection of those initial data $\xi_{0}\in H$ such that the system \eqref{nse:4} is globally well-posed has a $\mu$-measure greater than $1-\varepsilon$.
\end{theorem}

 The proof relies on probabilistic arguments developed in  \cite{FL21}, rough path theory tools from  \cite{hofmanova2019navier,HLN21} and a new Wong-Zakai approximation result, which itself combines probabilistic and rough path techniques.

The uniqueness of weak solutions to \eqref{nse:1} satisfying the energy inequality remains an outstanding open problem. A natural question is therefore whether there is a selection of  solutions such that the  semiflow property holds. Precisely, whether there exists a map $U:(u_{0},\mathbf Y)\mapsto u$ which associates a solution to the initial condition $u_{0}$, and the rough path $\mathbf Y$ so that
$$
U(u_{0},\mathbf Y)(t_{1}+t_{2})=U\big(U(u_{0},\mathbf Y)(t_{1}),\mathbf Y(t_{1}+\cdot)\big),\quad \text{for any }t_{1},t_{2}\geq 0.
$$
Without uniqueness, this property cannot hold true generally.
Moreover, if the driving signal $\mathbf Y$ is random and fulfills a suitable variant of the cocycle property in the rough path framework, we may further ask whether \eqref{nse:1} generates a random dynamical system over a measurable dynamical system $(\Omega,\mathcal{F},\mathbb{P},(\theta_{t})_{t\geq 0})$, meaning among others, whether the Borel measurable map
$$
\Phi:(t,\omega,u_{0})\mapsto U(u_{0},\mathbf Y(\omega))
$$
satisfies the cocycle property
$$
\Phi(t+s,\omega)=\Phi(t,\theta_{s}\omega)\circ\Phi(s,\omega),\quad\text{for any }s,t\in [0,\infty), \omega\in\Omega.
$$
Due to the difficulties originating from existence of the so-called energy sinks, these questions can only be answered affirmatively with the following modification. An auxiliary variable $E$ is included as part of the solution and it corresponds to the energy of the system. Accordingly, a weak solution consists of the couple $U=(u,E)$ of the velocity and the energy  and it has the initial datum $(u_{0},E_{0})$. With this modification, the above mappings become
$$
U:(u_{0},E_{0},\mathbf Y)\mapsto  (u,E)
,\qquad \Phi:(t,\omega,u_{0},E_{0})\mapsto U(u_{0},E_{0},\mathbf Y(\omega)).
$$
The semiflow and cocycle property are defined correspondingly.

\begin{theorem}
The system \eqref{nse:1} admits a semiflow selection in the class of weak solutions and, under appropriate assumptions on $\mathbf Y$, it generates a random dynamical system.
\end{theorem}

Let us finally note that the framework of URD was also used in \cite{DH23} to establish the convergence of a slow-fast system of coupled three dimensional Navier-Stokes equations where the fast component is perturbed by an additive $Q$-Wiener process $W$
\begin{align} \label{eq:system_NS}
\begin{cases}
du^\epsilon +(u^{\epsilon}+v^{\epsilon})\cdot u^{\epsilon} dt +\nabla p^{\epsilon}dt
= 
\nu \Delta u^\epsilon dt,
\\
\mbox{div}\, u^\epsilon  = 0,
\\
dv^\epsilon +(u^{\epsilon}+v^{\epsilon})\cdot v^{\epsilon} dt +\nabla q^{\epsilon}dt
= 
\epsilon^{-1} C v^\epsilon dt
+
\nu \Delta v^\epsilon dt
+\epsilon^{-1}
d{W}
\\
\mbox{div}\, v^\epsilon = 0.
\end{cases}
\end{align}
It was shown that the slow component $u^{\varepsilon}$ converges in law  towards a Navier-Stokes system
  with an It{\^o}-Stokes drift and a rough path driven transport noise. Interestingly, the limit driving rough path is identified as a geometric rough
  path which does not necessarily coincide with the Stratonovich lift of the
  Brownian motion.

\section{Euler equations}
\label{sec:Euler}

The work on the inviscid version of \eqref{nse:3}, i.e. Euler's equation, was initiated in \cite{CHLN1} where a general variational principle for fluid equations driven by rough paths is introduced. In \cite{CHLN1}, the main objective was to find a (rough) PDE which preserve a set of desired physical properties (known to hold for deterministic Euler) and for which the Lagrangian trajectories have the decomposition
\begin{equation} \label{eq:rough lagrangian}
d\phi_t(x) = u_t(\phi_t(x)) dt + \sigma_{\mu}(\phi_t(x)) d \mathbf{Y}^{\mu}_t , \qquad \phi_0(x) = x.
\end{equation}
The approach is motivated by Vladimir Arnold \cite{Arnold66} to study the Lagrangian trajectories of (deterministic) Euler's equation as paths with values in the diffeomorphism group
$$
\mathcal{D}_{\textrm{Vol}}(M) = \{ \psi : M \rightarrow M | \, \psi \textrm{ is a volume preserving diffeomorphism} \}
$$
where $M$ is a compact Riemannian manifold with volume form $d\textrm{Vol}(x)$. 
More precisely, the above is an infinite dimensional manifold whose tangent spaces can be described as
$$
T_{\psi} \mathcal{D}_{\textrm{Vol}}(M) = \{ X : M \rightarrow TM \, \big| \pi \circ X = \psi, \text{ and } \, \text{div}(X) = 0 \}
$$
If one equip this set with the (weak) Riemannian metric 
$$
\langle \langle X, Y \rangle \rangle_{\psi} = \int_{M} \langle X(x), Y(x) \rangle_{\psi(x)} d\textrm{Vol}(x),
$$
the insight of Arnold was that solutions of Euler's equation is (formally) in one-to-one correspondence with \emph{geodesics}  on the infinite-dimensional manifold. 

Starting from this observation, it is shown in \cite{CHLN1} that if one introduces the action functional
\begin{equation} \label{action functional}
S(u,\phi,\lambda) =  \frac12 \int_0^T \langle u_t, u_t \rangle_{L^2} dt  + \int_0^T \langle \lambda_t, d (\phi_t \circ \phi_t^{-1}) - u_t dt - \sigma_{\mu} d\mathbf{Y}_t^{\mu} \rangle_{L^2}
\end{equation}
(see \cite{CHLN1} for precise definition of this functional), then the critical points of $S$ satisfies the rough PDE 
\begin{align}\label{eq:euler}
\left\{
\begin{aligned}
du+(u\cdot\nabla u+\nabla p)d t +  (\sigma_{\mu}\cdot\nabla u + \nabla\sigma_{\mu}u )d\mathbf Y^{\mu}  &= 0 ,\qquad \text{on }[0,T]\times M,\\
{\rm div} u &= 0,
\end{aligned}
\right.
\end{align}
and its Lagrangian trajectories satisfies \eqref{eq:rough lagrangian}. In fact, in \eqref{action functional}, the term $\lambda$ should be thought of as a Lagrangian multiplier which forces the constraint \eqref{eq:rough lagrangian}. Moreover, several physical properties (Kelvin's circulation theorem, enstrophy balance) are at least formally preserved. 

\subsection{Smooth solutions}

In \cite{CHLN2} well-posedness of \eqref{eq:euler} is studied when $M = \mathbb{T}^d$ on the scale $(W^{k,2}(\mathbb{T}^d))$
and the initial data $u_0$ belongs to the space $H^{k}$ for $k \geq k_{\star} : = \lfloor \frac{d}{2} \rfloor + 2$. The relatively smooth initial data ensures that one can avoid the tensorization argument typically employed to show energy estimates - instead a direct testing of the solution against itself yields $L^2$ estimates.

\begin{theorem}[Local well-posedness]\label{thm:local_wp} 
Assume that $u_0\in H^k$ and $\sigma_{\mu} \in (W_{\text{div-free}}^{k+2})$ for all $\mu \in \{1, \dots, m\}$ for some $k \geq k_{\star}$.
Then there exists a unique maximally extended $H^{k}$-solution $u\in C_w([0,T_{\textnormal{max}});H^{k})\cap C^{p-\textnormal{var}}([0,T_{\textnormal{max}});H^{k-1})$ of \eqref{eq:euler}. on the interval $[0,T_{\textnormal{max}})$. The time $T_{\textnormal{max}}$ is uniquely specified by the property if 
$$
T_{\textnormal{max}}<\infty, \qquad \Longleftrightarrow \qquad \limsup_{t\uparrow T_{\textnormal{max}}} |u_t|_{H^{k_{\star}}}=\infty.
$$ 
\end{theorem}

We note that the above maximal time $T_{\max}$ depends only on the norm of the solution in the $H^{k_{\star}}$-norm, which shows that the solution cannot be extended in time by choosing more regular initial condition. This results is sometimes referred to as a ``no loss/no gain" result.

In addition, a more precise description of the blow-up in terms of the vorticity $\xi = \nabla \times u$, known as the Beale-Kato-Majda criterion, was proved.

\begin{theorem}[BKM blow-up criterion]\label{thm:BKM}
Assume that $u_0\in H^k$ and  $\sigma_{\mu} \in W_{\text{div-free}}^{k_{\star}+4,\infty}$ if $m=m_*$ and $\sigma_{\mu} \in W_{\text{div-free}}^{k+2}$ if $k>k_{\star}$.  
Then 
\begin{equation} \label{BKM}
T_{\textnormal{max}}<\infty \qquad \Longleftrightarrow \qquad \int_0^{T_{\textnormal{max}}} |\xi_t|_{L^\infty} \,d t =\infty.
\end{equation}
\end{theorem}

When $d=2$ the vorticity equation takes the form of a scalar transport equation
\begin{equation} \label{eq:vorticity 2d}
d \xi_t + u_t \cdot \nabla \xi_t dt  + \sigma_{\mu} \cdot \nabla \xi_t d\mathbf{Y}^{\mu}_t = 0.
\end{equation}
Now, since $u$ solves \eqref{eq:euler} it is divergence free. By assumption $\sigma_{\mu}$ is divergence free for every $\mu \in \{1,2 , \dots, m\}$ and so we expect \eqref{eq:vorticity 2d} to be a conservation law. In fact, in \cite{CHLN2} it is shown that if we assume $\xi_0 \in L^{\infty}$ we have 
\begin{equation}\label{eq:Lp conservation}
|\xi_t|_{L^p}=|\xi_0|_{L^p}.
\end{equation}
for all $p \in [2,\infty]$. From this and \eqref{BKM} we obtain the following corollary.

\begin{corollary} \label{cor:2d wp}
Assume $d=2$, $\xi_0 \in L^{\infty}$, $u_0\in H^k$ and  $\sigma_{\mu} \in W_{\text{div-free}}^{k_{\star}+4,\infty}$ if $m=m_*$ and $\sigma_{\mu} \in W_{\text{div-free}}^{k+2}$ if $k>k_{\star}$. Then $T_{\max} = \infty$. 
\end{corollary}

\subsection{Yudovich solutions}

Corollary \ref{cor:2d wp} and \eqref{eq:Lp conservation} hints at the possibility of a rough path analogue of Yudovich theory, i.e. well-posedness of solutions of \eqref{eq:vorticity 2d} in the class of $L^{\infty}$. This was first proved in \cite{roveri} and 
the main result is summarized in the following theorem.

\begin{theorem}
Assume $\xi_0 \in L^{\infty}(\T^2)$ and that $\sigma_{\mu} \in C^3_b$ and $\text{div}\sigma_{\mu} = 0$ for $\mu \in \{1,\dots,m\}$. Then \eqref{eq:vorticity 2d} is well posed in the class of bounded $\xi: [0,T] \times \T^2 \rightarrow \T^2$
Moreover, the solution admits a Lagrangian representation, namely $\xi_t = (\phi_t)_{\sharp} \xi_0$ where 
$$
d\phi_t(x) = u_t(\phi_t(x)) dt + \sigma_{\mu}(\phi_t(x)) d\mathbf{Y}_t^{\mu}.
$$
Finally, the solution depends continuously on the data $(\xi_0,\sigma,\mathbf{Y}) \in (L^{\infty},(C_b^3)^m,\mathscr{C}^{\alpha}_g) $ as an $(L^{\infty})_{w^{\ast}}$-valued map .
\end{theorem}
Let us remark that in the stability result, the initial condition $\xi_0$ is taken from $L^{\infty}$ with the strong topology, whereas the solution $\xi$ takes values in $L^{\infty}$ equipped with the weak topology.
Furthermore, let us remark on the above Lagrangian representation; the boundedness of $\xi$ is enough to ensure that the velocity $u = (\text{curl})^{-1} \xi$ is log-Lipschitz, viz
$$
|u_t(x) - u_t(\tilde{x})| \leq h(|x - \tilde{x}|)
$$
where $h(r) = r(1-\log(r))1_{[0, \frac{1}{e})}(r) + (r + \frac{1}{e}) 1_{[\frac{1}{e},\infty)}(r)$. In the deterministic theory it is well-known that log-Lipschitz vector fields is enough to guarantee a well-posed induced flow. In \cite{roveri} it is shown that this can also be extended to the RDE setting. 

Simultaneously and independently, the paper \cite{GLN} introduced a more general class of non-linear continuity equations on the form
\begin{equation} \label{eq:continuity}
d \xi_t + \text{div}( (K \ast \xi_t)  \xi_t) dt  +\text{div}( \sigma_{\mu} \xi_t) d\mathbf{Y}^{\mu}_t = 0, \qquad  \xi|_{t=0} \in L^1(\R^d) \cap L^{\infty}(\R^d).
\end{equation}
The equation is on the (technically more challenging) underlying state space to be $\R^d$. Above $K$ is a suitable convolution kernel which in particular has the smoothing effect
$$
|K \ast f(x) - K \ast f(\tilde{x})| \leq h(|x - \tilde{x}|) (\|f\|_{L^1(\R^d)} + \|f\|_{L^{\infty}(\R^d)})
$$
for an \emph{Osgood} modulus of continuity $h$, i.e. a function $h:[0,\infty) \rightarrow [0,\infty)$ which is increasing, subadditive and such that 
$$
\int_0^{\epsilon}  \frac{1}{h(r)} dr = \infty.
$$
In particular, when $d=2$, the Biot-Savart kernel satisfies these assumptions, so \eqref{eq:continuity} covers the Euler's equation in vorticity form on $\mathbb{R}^2$. 
The main result of \cite{GLN} is as follows.
\begin{theorem}
Assume $\xi_0 \in L^1(\R^d) \cap L^{\infty}(\R^d)$ and that $\sigma_{\mu} \in C^3_b$ and $\text{div}\sigma_{\mu} = 0$ for $\mu \in \{1,\dots,m\}$. Then \eqref{eq:continuity} is well posed in the class of $\xi \in L^{\infty}( [0,T]; L^1(\R^d) \cap L^{\infty}(\R^d))$
Moreover, the solution admits a Lagrangian representation
$$
d\phi_t(x) = u_t(\phi_t(x)) dt + \sigma_{\mu}(\phi_t(x)) d\mathbf{Y}_t^{\mu}.
$$
where $u_t = K \ast \xi_t$. 
Finally, on bounded sets of $L^1(\R^d) \cap L^{\infty}(\R^d)$, we have sequential stability w.r.t. the initial data in weak and strong topologies.
\end{theorem}
The stability results of \cite{GLN} are natural from the perspective of dynamical systems. Applied to the particular case of $d=2$ and when $K$ is the Biot-Savart kernel, \cite{GLN} obtains the following result.
\begin{theorem}
Euler's equation \eqref{eq:vorticity 2d} generates a flow on the set 
$$
\chi_R := \{ \xi_0 \in L^1(\R^2) \cap L^{\infty}(\R^2) : \| \xi_0 \|_{L^1(\R^2)} +  \| \xi_0 \|_{L^{\infty}(\R^2)} \leq R \}.
$$
When $\omega \mapsto \mathbf{Y}(\omega)$ is a random rough path with the rough path cocycle property, the solution mapping 
$$
(t,\omega,\xi_0) \mapsto  \xi_t(\xi_0, \mathbf{Y}(\omega))
$$
defines a continuous random dynamical system when $\chi_R$ is equipped with either the weak$^{\ast}$ topology induced by $L^{\infty}$ or the strong topology induced by $L^p$ for $p \in (1,\infty)$. 
\end{theorem} 

\appendix 
\section{Sewing lemma}

Herein we call \textit{control} a map \( w\colon \Delta\to [0,\infty) \) which is continuous, vanishing at the diagonal \( \{(t,t):t\in [0,T]\} \) and superadditive, namely
\[
w(s,\theta) + w(\theta,t)\le w(s,t),\quad 
\text{for each }(s,\theta,t)\in \Delta_2.
\]

We recall the statement of Gubinelli's sewing lemma in a Banach space $(E,|\cdot|)$, as formulated for instance in \cite{gubinelli2010rough}.

\begin{theorem}[Sewing Lemma]
	\label{thm:sewing}
	Let $H \colon \Delta \rightarrow E$ and $C>0$ be such that 
	\begin{equation}
	\label{a_gamma}
	\left|\delta H_{s\theta t}\right|\leq Cw (s,t)^{a }\,
	, \quad 0 \leq s \leq \theta \leq t \leq T
	\end{equation}
	for some $a > 1$, and some control function $w ,$ and denote by $[\delta H]_{a,w }$ the smallest possible constant $C$ in the previous bound. 
	
	There exists a unique pair $I\colon [0,T] \rightarrow E$ and $I^{\natural} : \Delta \rightarrow E$ satisfying
	\[
	I_{t}-I_s = H_{st} + I_{st}^{\natural}
	\]
	where for $0\leq s\leq t\leq T$,
	\[
	|I_{st}^{\natural}| \leq C_a [\delta H]_{a,w } w (s,t)^{a}\,,
	\]
	for some constant $C_a$ only depending on $a$. In fact, $I$ is defined via the Riemann type integral approximation
	\begin{equation} \label{RiemannSum}
	I_t = \lim \sum_{ i=1 }^n H_{t^n_i t^n_{i+1}} \,,
	\end{equation}
	the above limit being taken along any sequence of partitions $\{t^n,n\geq 0\}$ of $[0,t]$ whose mesh-size converges to $0$.
\end{theorem}

\section{Rough Gronwall}

We recall the following useful inequality (see \cite{hofmanova2018rough,deya2019priori,hocquet2018energy} for a proof) which happens to be a $p$-variation analogue of Gronwall Lemma.
\begin{lemma}[Rough Gronwall lemma]
	\label{lem:gronwall}
	Let $E\colon[0,T]\to [0,\infty)$ be a path such that there exist constants $\kappa,\ell>0,$ a super-additive map $\varphi $ and a control $w $ such that:
	\begin{equation}\label{rel:gron}
		\delta E_{s,t}\leq \left(\sup_{s\leq r\leq t} E_r\right)w (s,t)^{\kappa }+\varphi (s,t)\, ,
	\end{equation}
	for every $(s,t)$ under the smallness condition $w (s,t)\leq \ell$.
	
	Then, there exists a constant $\tau _{\kappa ,\ell}>0$ such that
	\begin{equation}
		\label{concl:gron}
		\sup_{0\leq t\leq T}E_t\leq \exp\left(\frac{w (0,T)}{\tau _{\kappa ,\ell}}\right)\left[E_0+\sup_{0\leq t\leq T}\left|\varphi (0,t)\right|\right].
	\end{equation}
\end{lemma}

\bibliographystyle{plain}
\bibliography{biblio_urd}

\begin{thebibliography}{10}

\bibitem{Arnold66}
V.~I. Arnold.
\newblock Sur la geometrie differentielle des groupes de lie de dimension
  infinie et ses applications a lhydrodynamique des fluides parfaits.
\newblock {\em Ann. Inst. Fourier (Grenoble)}, 16:319--361, 1966.

\bibitem{bailleul2017unbounded}
Isma{\"e}l Bailleul and Massimiliano Gubinelli.
\newblock {Unbounded rough drivers}.
\newblock {\em Annales Math{\'e}matiques de la Facult{\'e} des Sciences de
  Toulouse}, \textbf{26}(4), 2017.

\bibitem{CN}
Michele Coghi and Torstein Nilssen.
\newblock Rough nonlocal diffusions.
\newblock {\em Stochastic Processes and their Applications}, 141:1--56, 2021.

\bibitem{CHLN2}
Dan Crisan, Darryl~D. Holm, James-Michael Leahy, and Torstein Nilssen.
\newblock Solution properties of the incompressible euler system with rough
  path advection.
\newblock {\em Journal of Functional Analysis}, 283(9):109632, 2022.

\bibitem{CHLN1}
Dan Crisan, Darryl~D. Holm, James-Michael Leahy, and Torstein Nilssen.
\newblock Variational principles for fluid dynamics on rough paths.
\newblock {\em Advances in Mathematics}, 404:108409, 2022.

\bibitem{DPZ}
Giuseppe {Da Prato} and Jerzy Zabczyk.
\newblock {\em {Stochastic equations in infinite dimensions}}.
\newblock Cambridge University Press, 2008.

\bibitem{DH23}
Arnaud Debussche and Martina Hofmanov{\'a}.
\newblock Rough analysis of two scale systems.
\newblock {\em arXiv:2306.15781}, 2023.

\bibitem{deya2019priori}
Aur{\'e}lien Deya, Massimiliano Gubinelli, Martina Hofmanova, and Samy Tindel.
\newblock A priori estimates for rough pdes with application to rough
  conservation laws.
\newblock {\em Journal of Functional Analysis}, 276(12):3577--3645, 2019.

\bibitem{diehl2017stochastic}
Joscha Diehl, Peter~K Friz, and Wilhelm Stannat.
\newblock Stochastic partial differential equations: a rough paths view on weak
  solutions via feynman--kac.
\newblock {\em Annales de la Facult{\'e} des sciences de Toulouse:
  Math{\'e}matiques}, 26(4):911--947, 2017.

\bibitem{diperna1989ordinary}
Ronald~J DiPerna and Pierre-Louis Lions.
\newblock {Ordinary differential equations, transport theory and Sobolev
  spaces}.
\newblock {\em Inventiones mathematicae}, {98}(3):511--547, 1989.

\bibitem{FHLN20}
Franco Flandoli, Martina Hofmanov\'{a}, Dejun Luo, and Torstein Nilssen.
\newblock Global well-posedness of the 3{D} {N}avier-{S}tokes equations
  perturbed by a deterministic vector field.
\newblock {\em Ann. Appl. Probab.}, 32(4):2568--2586, 2022.

\bibitem{FL21}
Franco Flandoli and Dejun Luo.
\newblock High mode transport noise improves vorticity blow-up control in 3{D}
  {N}avier-{S}tokes equations.
\newblock {\em Probab. Theory Related Fields}, 180(1-2):309--363, 2021.

\bibitem{FNS}
Peter~K. Friz, Torstein Nilssen, and Wilhelm Stannat.
\newblock Existence, uniqueness and stability of semi-linear rough partial
  differential equations.
\newblock {\em Journal of Differential Equations}, 268(4):1686--1721, 2020.

\bibitem{FrizVictoir}
Peter~K Friz and Nicolas~B Victoir.
\newblock {\em Multidimensional stochastic processes as rough paths: theory and
  applications}, volume 120.
\newblock Cambridge University Press, 2010.

\bibitem{GLN}
Lucio Galeati, James-Michael Leahy, and Torstein Nilssen.
\newblock On the well-posedness of (nonlinear) rough continuity equations.
\newblock {\em arXiv:2502.04982}, 2025.

\bibitem{GNS}
Erlend Grong, Torstein Nilssen, and Alexander Schmeding.
\newblock Geometric rough paths on infinite dimensional spaces.
\newblock {\em Journal of Differential Equations}, 340:151--178, 2022.

\bibitem{gubinelli2010rough}
Massimiliano Gubinelli and Samy Tindel.
\newblock {Rough evolution equations}.
\newblock {\em The Annals of Probability}, {38}(1):1--75, 2010.

\bibitem{G22}
Emanuela Gussetti.
\newblock On ergodic invariant measures for the stochastic
  landau-lifschitz-gilbert equation in 1d.
\newblock {\em arXiv:2208.02136}, 2022.

\bibitem{G23}
Emanuela Gussetti.
\newblock Pathwise central limit theorem and moderate deviations via rough
  paths for spdes with multiplicative noise.
\newblock {\em arXiv:2307.10965}, 2023.

\bibitem{gussetti2023pathwise}
Emanuela Gussetti and Antoine Hocquet.
\newblock A pathwise stochastic landau-lifshitz-gilbert equation with
  application to large deviations.
\newblock {\em Journal of Functional Analysis}, page 110094, 2023.

\bibitem{hocquet2015landau}
Antoine Hocquet.
\newblock {\em The Landau-Lifshitz-Gilbert equation driven by Gaussian noise}.
\newblock PhD thesis, \'Ecole Polytechnique, 2015.

\bibitem{hocquet2021quasilinear}
Antoine Hocquet.
\newblock Quasilinear rough partial differential equations with transport
  noise.
\newblock {\em Journal of Differential Equations}, 276:43--95, 2021.

\bibitem{hocquet2018energy}
Antoine Hocquet and Martina Hofmanov{\'a}.
\newblock An energy method for rough partial differential equations.
\newblock {\em Journal of Differential Equations}, 265(4):1407--1466, 2018.

\bibitem{hocquet2020ito}
Antoine Hocquet and Torstein Nilssen.
\newblock An it{\^o} formula for rough partial differential equations and some
  applications.
\newblock {\em Potential Analysis}, pages 1--56, 2020.

\bibitem{hocquet2020generalized}
Antoine Hocquet, Torstein Nilssen, and Wilhelm Stannat.
\newblock Generalized burgers equation with rough transport noise.
\newblock {\em Stochastic Processes and their Applications}, 130(4):2159--2184,
  2020.

\bibitem{hofmanova2018rough}
Martina Hofmanov{\'a}.
\newblock On the rough gronwall lemma and its applications.
\newblock In {\em Stochastic Partial Differential Equations and Related Fields:
  In Honor of Michael R{\"o}ckner SPDERF, Bielefeld, Germany, October 10-14,
  2016 1}, pages 333--344. Springer, 2018.

\bibitem{hofmanova2019navier}
Martina Hofmanov{\'a}, James-Michael Leahy, and Torstein Nilssen.
\newblock On the navier--stokes equation perturbed by rough transport noise.
\newblock {\em Journal of Evolution Equations}, 19(1):203--247, 2019.

\bibitem{HLN21}
Martina Hofmanov\'{a}, James-Michael Leahy, and Torstein Nilssen.
\newblock On a rough perturbation of the {N}avier-{S}tokes system and its
  vorticity formulation.
\newblock {\em Ann. Appl. Probab.}, 31(2):736--777, 2021.

\bibitem{ladyzhenskaya1968linear}
O.~Ladyzhenskaya, V.~Solonnikov, and N.~Uraltseva.
\newblock {Linear and quasilinear parabolic equations of second order}.
\newblock {\em Translation of Mathematical Monographs, AMS, Rhode Island},
  1968.

\bibitem{lunardi2009interpolation}
Alessandra Lunardi.
\newblock {\em Interpolation theory}.
\newblock Edizioni della normale, 2009.

\bibitem{moser1964harnack}
J{\"u}rgen Moser.
\newblock A harnack inequality for parabolic differential equations.
\newblock {\em Communications on pure and applied mathematics}, 17(1):101--134,
  1964.

\bibitem{roveri}
Leonardo Roveri and Francesco Triggiano.
\newblock Well-posedness of rough 2d euler equation with bounded vorticity.
\newblock {\em arXiv:2410.24040}, 2024.

\bibitem{triebel1983theory}
Hans Triebel.
\newblock {\em Theory of Function Spaces}.
\newblock Springer Science \& Business Media, 1983.

\end{thebibliography}
\end{document}